\documentclass[11pt,twoside]{article}
\usepackage[letterpaper,top=1.15in,bottom=1.0in,left=1.2in,right=1.2in,headheight=14pt,headsep=0.22in]{geometry}
\usepackage[T1]{fontenc}
\usepackage{lmodern}
\usepackage{amsmath,amssymb,amsthm,mathtools}
\usepackage{booktabs,array,longtable,graphicx,microtype}
\usepackage[hidelinks]{hyperref}
\usepackage{enumitem}
\usepackage{fancyhdr}
\usepackage{titling}
\usepackage{titlesec}
\usepackage{needspace}
\usepackage{placeins}
\usepackage{tikz}
\titleformat{\section}{\normalfont\large\scshape\centering}{\thesection.}{0.5em}{}
\titleformat{\subsection}{\normalfont\normalsize\bfseries}{\thesubsection.}{0.5em}{}
\pretitle{\begin{center}\large\bfseries\MakeUppercase}
\posttitle{\par\end{center}\vskip 1.0em}
\preauthor{\begin{center}\normalsize}
\postauthor{\par\end{center}\vskip 0.8em}
\predate{}
\postdate{}
\numberwithin{equation}{section}
\hypersetup{pdftitle={Compatible additions on a six-element commutative semigroup},pdfauthor={}}
\allowdisplaybreaks
\newtheorem{theorem}{Theorem}[section]
\newtheorem{proposition}[theorem]{Proposition}
\newtheorem{lemma}[theorem]{Lemma}
\newtheorem{corollary}[theorem]{Corollary}
\theoremstyle{definition}

\theoremstyle{remark}
\newtheorem{remark}[theorem]{Remark}
\newcommand{\V}{\mathsf V}
\newcommand{\Sub}{\operatorname{Sub}}

\newcommand{\F}{\mathcal F}
\newcommand{\Can}{\operatorname{Can}}
\newcommand{\gammaD}{\gamma^{D}}

\newcommand{\ManuscriptAuthors}{}

\renewcommand{\ManuscriptAuthors}{%
  Lili Wang, Jinjing Wu, Aifa Wang \\[5pt]
  {\normalsize School of Mathematical Sciences}\\
  {\normalsize Chongqing University of Technology}\\
  {\normalsize Chongqing, 400054, P.R. China}\\[3pt]
  {\small Corresponding author: Aifa Wang, \href{mailto:wangaf@cqut.edu.cn}{\texttt{wangaf@cqut.edu.cn}}}%
}

\hypersetup{pdfauthor={ Lili Wang; Jinjing Wu; Aifa Wang}}
\title{Compatible additions on a six-element commutative semigroup: equational bases and subvariety lattices}
\author{\ManuscriptAuthors}
\date{}
\fancypagestyle{plain}{%
  \fancyhf{}
  \fancyfoot[C]{\thepage}
  }
\renewenvironment{abstract}{%
  \begin{center}\begin{minipage}{0.82\textwidth}\small
  \textsc{Abstract.}\ }{%
  \end{minipage}\end{center}\vspace{0.7em}}
\renewcommand{\footnoterule}{\kern-3pt\hrule width 0.28\columnwidth\kern 2.6pt}
\begin{document}
\maketitle
\begin{abstract}
Let $M$ be the six-element commutative semigroup occurring as the common multiplicative reduct of the semirings $SR_6$ and $TR_6$. The closing paragraph of Shao, Ren, and Gao~\cite{ShaoRenGao2026} asks for the finite-basis and subvariety questions for the four remaining compatible additions on $M$. We answer these questions for the four isomorphism types $R_{01},R_{02},R_{11},R_{12}$. First, we classify all compatible additions on $M$: there are nine labelled additions and six isomorphism types, parametrized by $R_{ij}$ with $0\leq i\leq j\leq 2$. For each of the four new types we give a graph-theoretic criterion for every identity, an explicit infinite basis, and a proof of nonfinite basability. The generated varieties $\V(R_{01})$ and $\V(R_{02})$ have eleven subvarieties each, while $\V(R_{11})$ has sixty-six. The lattice $\Sub(\V(R_{12}))$ is countably infinite. Every identity in this variety reduces to a subset of twenty-five fixed identities together with two monotone path families $\gamma_n$ and $\gammaD_n$. This yields a canonical signature $(H,p,q)$, complete normal forms, explicit meet and join operations, and a formula for all covers. There are 153 fixed nodes, 43 one-parameter families, and 9 two-parameter families; exactly eighteen subvarieties are finitely based, and the unique limit subvariety is $\V(SR_6)$. The strong nonfinite-basis status of the four finite semirings remains open.
\end{abstract}

\begingroup
\renewcommand{\thefootnote}{}
\begin{NoHyper}
\footnotetext{\textit{2020 Mathematics Subject Classification.} 16Y60, 03C05, 08B15, 08B05.\\
\textit{Key words and phrases.} additively idempotent semiring, compatible addition, finite basis problem, limit variety, subvariety lattice, graph semiring.}
\end{NoHyper}
\endgroup

\section{Introduction}
A variety is a class of algebras closed under subalgebras, homomorphic images, and arbitrary direct products. By Birkhoff's theorem, varieties are precisely equational classes. The finite basis problem asks whether the identities of an algebra can be generated by finitely many identities. This problem is especially subtle for additively idempotent semirings (ai-semirings), where the additive operation is a commutative idempotent semigroup and the multiplicative operation is a semigroup satisfying both distributive laws.

The max-plus algebra and its finite quotients provide a rich source of examples; see \cite{AcetoEsikIngolf2023,Dolinka2007,JacksonRenZhao2022}. A variety is called a limit variety if it is nonfinitely based while every proper subvariety is finitely based. Lyu, Ren, and Yue \cite{LyuRenYue2026} proved that the variety generated by the six-element semiring $SR_6$ is a limit variety and described its four-element subvariety chain. Shao, Ren, and Gao \cite{ShaoRenGao2026} studied a second semiring $TR_6$ on the same multiplicative reduct, proved that $SR_6$ and $TR_6$ are nonfinitely based but not strongly nonfinitely based, and determined the nine-element subvariety lattice of $\V(SR_6,TR_6)$. For related work on limit varieties and finitely based reducts, see \cite{RenJacksonZhao2023,RenZhao2016}.

The common multiplicative reduct has more compatible additions. The authors of~\cite{ShaoRenGao2026} observed that there are six additions up to isomorphism, studied two of them, and left the other four for future work. The purpose of this paper is to carry out that programme. We use the notation $R_{ij}$ introduced below; $R_{00}=SR_6$ and $R_{22}=TR_6$, whereas $R_{01},R_{02},R_{11},R_{12}$ are the four remaining types.

Our results are as follows.
\begin{enumerate}[label=(\roman*),leftmargin=2.2em]
\item We classify all compatible additions on the fixed multiplication, obtaining nine labelled tables and six isomorphism classes.
\item We give a uniform graph model for terms and a complete identity criterion for each of the four new semirings. Each has an explicit infinite basis and is nonfinitely based.
\item We determine the complete subvariety lattices for $R_{01},R_{02}$, and $R_{11}$ (of sizes $11,11$, and $66$), and give a canonical description of the whole countably infinite lattice $\Sub(\V(R_{12}))$.
\item The latter lattice is parametrized by a finite set $H$ of fixed identities and two path thresholds $0\le p\le q\le\infty$. We prove effective formulas for normalization, inclusion, meets, joins, and covers. Exactly eighteen nodes are finitely based; all other nodes are nonfinitely based, and the unique limit subvariety is $\V(SR_6)$.
\end{enumerate}

The proofs use graph folding, closed-term algebras and path-distance invariants. A finite implication calculus describes all relations among the fixed representatives. Separating the resulting signatures into ordinary edge states, conditioned edge states and marking states gives direct counting proofs for the six threshold regions and the three finite lattices. The same calculus identifies the sixteen-node distributive interval above the internal-mark three-edge variety.

\section{The common multiplication and the six additions}
We work in the language with two binary operation symbols $+$ and $\cdot$, without constants. To avoid confusion with a multiplicative identity, we relabel the six elements as $0,c,p,q,a,b$; the symbol $0$ is only a name for the multiplicative zero and additive maximum of the displayed reduct, not a language constant. The only nonzero products are
\[
pb=bp=ab=ba=aq=qa=c,
\]
and every other product, including every square, is $0$. Thus the graph of nonzero multiplication is the path
\[
p\;--\;b\;--\;a\;--\;q.
\]
The element $0$ is absorbing for multiplication and is the greatest element for every compatible addition.

For $s\in M$ put $N(s)=\{t\in M:st=c\}$. Then
\[
N(0)=N(c)=\varnothing,\quad N(p)=\{b\},\quad N(q)=\{a\},\quad
N(a)=\{b,q\},\quad N(b)=\{p,a\}.
\]
If $+$ is compatible, then $(a+b)^2=0$ and distributivity give $0+c=0$. Consequently distributivity is equivalent to
\[
N(s+t)=N(s)\cap N(t).
\]
This observation makes the addition classification elementary.

\begin{theorem}[classification of compatible additions]\label{thm:additions}
There are nine compatible additions on the labelled multiplication, and six isomorphism types. They are denoted by $R_{ij}$, $i,j\in\{0,1,2\}$, with $R_{ij}\cong R_{ji}$ and no other identifications. In $R_{ij}$ the additive order is the transitive closure of
\[
a<p<0,\qquad b<q<0,\qquad c<0,
\]
together with the first $i$ elements of the chain $a<p$ lying below $c$ and the first $j$ elements of the chain $b<q$ lying below $c$.
\end{theorem}

\begin{proof}
The neighbourhood-intersection equation gives $a+p=p$ and $b+q=q$, while all other sums of distinct nonzero elements can only be $0$ or $c$. If $s+0=c$, associativity and $c+0=0$ would imply $c=0$, so $0$ is the additive maximum. On each of the chains $a<p$ and $b<q$, the elements below $c$ form a downward closed prefix, giving three choices. For elements from different chains, the sum is $c$ exactly when both are below $c$, and is $0$ otherwise. Each displayed order is a join semilattice, and the neighbourhood-intersection equation gives both distributive laws. A multiplication automorphism fixes $0,c$ and either fixes the path or reflects it, interchanging $a\leftrightarrow b$ and $p\leftrightarrow q$. Hence $(i,j)$ and $(j,i)$ are the only identifications.
\end{proof}

The original notation is $SR_6=R_{00}$ and $TR_6=R_{22}$. The four additions studied in this paper are therefore $R_{01},R_{02},R_{11},R_{12}$. For reference, the defining order relations not common to all four are shown in Table~\ref{tab:orders}.

\begin{table}[ht]
\centering
\caption{Additional additive comparisons for the four new types.}
\label{tab:orders}
\begin{tabular}{c|l}
\toprule
Type & Comparisons below $c$\\
\midrule
$R_{01}$ & $b<c$\\
$R_{02}$ & $b<q<c$\\
$R_{11}$ & $a<c$ and $b<c$\\
$R_{12}$ & $a<c$ and $b<q<c$\\
\bottomrule
\end{tabular}
\end{table}

\section{A graph calculus for identities}
For terms in an additively idempotent semiring, repeated summands may be deleted and multiplication distributes over addition. We write
\[
u\preceq v\quad\Longleftrightarrow\quad u+v\approx v.
\]
An identity $u\approx v$ is equivalent to the finite set of absorption inequalities obtained by comparing each monomial on either side with the other side. Write $M(u)$ for the variables occurring as linear summands and $G(u)$ for the graph whose edges are the square-free quadratic summands. Its vertex set consists only of variables incident with an edge. The marked set is
\[
W(u)=M(u)\cap V(G(u)).
\]
Isolated linear summands, namely $M(u)\setminus V(G(u))$, are retained separately and are never treated as marked vertices of the quadratic graph.

A term is called basic-zero if it contains a square, a monomial of degree at least three, or an odd cycle in its quadratic graph. In each $R_{ij}$ every basic-zero term evaluates to $0$. Otherwise the quadratic graph is bipartite. In a connected component, let $E$ and $O$ denote its two parts, with $E$ chosen to contain the marked vertices whenever this is possible. Let $E(W)$ (respectively $O(W)$) be the vertices at even (respectively odd) distance from $W$; a path of length zero is allowed in $E(W)$.

The following theorem is the basic identity criterion. It is stated for a nonzero right-hand term; if the right side is basic-zero, every absorption inequality is valid.

\begin{theorem}[identity criteria]\label{thm:criterion}
Let $u$ not be basic-zero and let $G$ be its quadratic graph. Monomials already displayed in $u$ are always absorbed. Apart from them, the following list is exhaustive.
\begin{enumerate}[label=(\alph*),leftmargin=2.2em]
\item In $R_{01}$, if two marked vertices are joined by an odd path, then $u=0$. Otherwise the only nontrivial absorbed monomials are edges $xy$ with one endpoint in $W$ and the other in $O(W)$.
\item In $R_{02}$, if two marked vertices are joined by an odd path, then $u=0$. Otherwise the only nontrivial absorbed monomials are the degree-one variables in $E(W)$.
\item In $R_{11}$, a nontrivial edge $xy$ is absorbed exactly when $x$ and $y$ are joined by an odd path and at least one endpoint is marked.
\item In $R_{12}$, a nontrivial edge $xy$ is absorbed exactly when $x$ and $y$ are joined by an odd path and both endpoints are marked.
\end{enumerate}
\end{theorem}

\begin{proof}
If $G$ has an edge and $u\ne0$, every quadratic summand has value $c$, and therefore $u=c$. The restriction of the valuation to $V(G)$ is a homomorphism to $P_4=p-b-a-q$. Conversely, any such homomorphism extends to a valuation with $u=c$ if the marked vertices have images in the following sets and the isolated linear summands have images at most $c$:
\[
\begin{array}{c|c}
R_{01}&\{b\}\\
R_{02}&\{b,q\}\\
R_{11}&\{a,b\}\\
R_{12}&\{a,b,q\}.
\end{array}
\]
For $R_{01}$ and $R_{02}$ the permitted marked images lie in the same part of $P_4$. Thus opposite-part marks in one component force $u=0$. Otherwise, in $R_{01}$ all marks have image $b$, and vertices at odd distance from a mark have image $a$ or $p$; both are adjacent to $b$. This proves (a), including edges between different marked components. In $R_{02}$ every vertex at even distance from a mark has image $b$ or $q$, both at most $c$, proving (b). In $R_{11}$ a marked endpoint has image $a$ or $b$ and is adjacent to every vertex in the opposite part. In $R_{12}$, after orienting the parts, the marked images in one part consist only of $a$, and those in the other part lie in $\{b,q\}$; all these pairs are adjacent. This proves the asserted positive edge cases.

We give countervaluations for the remaining cases. An undisplayed linear query outside $V(G)$ is either a fresh variable, which may be sent to $0$, or an already displayed isolated summand. For an unmarked graph vertex, start with a two-colouring into the middle edge and move the queried vertex to an end of $P_4$. In $R_{01}$ use $q$ on the marked side and $p$ on the other side; in $R_{02}$ an unforced vertex can be put on the $a$ side; in $R_{11}$ and $R_{12}$ orient the component so that the queried vertex can be sent to $p$. These images are not at most $c$. An unmarked component can always be oriented independently.

For quadratic queries involving an isolated linear variable, send that variable to $c$; for a fresh variable use $0$. Queries with endpoints in the same part can be made zero by mapping both into that part of $P_4$. For a missing opposite-part edge with both endpoints unmarked, preserve the endpoints as $p,q$ and send all other vertices in the two parts to $a,b$, respectively. Since the queried edge is missing, every actual edge still has nonzero product. This assignment respects the marked restrictions in $R_{11}$ and $R_{12}$. In $R_{12}$, if one endpoint is marked, instead send it to $q$, the unmarked endpoint to $p$, and the remaining vertices of their parts to $b,a$. Again all actual edges have nonzero product.

For $R_{01}$, orient each marked component with its marks on the $b$ side. If a query is not one of the positive cases in (a), an endpoint on this side that must be moved to $q$ is unmarked; an opposite endpoint can be moved to $p$. Missingness of the query prevents the only forbidden pair $pq$ from occurring as an actual edge. If both endpoints lie on corresponding sides of different components, the middle-edge colouring itself makes their product zero. The same construction works for $R_{02}$ even when the endpoint moved to $q$ is marked, because $q\le c$ there. In either semiring an unmarked component may be oriented to put the two queried endpoints on corresponding sides. For $R_{11},R_{12}$, independently orienting distinct components also makes a cross-component query zero. These cases exhaust the undisplayed quadratic queries. Squares and higher products are always zero.

Finally, if $G$ has no edges, assign all displayed variables to $c$. Every quadratic or higher query then has value $0$, and a fresh linear variable may also be assigned $0$. Thus a purely linear term absorbs precisely its displayed summands. This covers the case excluded from the homomorphism argument.
\end{proof}

The graph formulation is useful for two reasons. It gives a short proof of the explicit bases below, and it remains stable under arbitrary polynomial substitutions once one uses finite supports: a nonzero image of a quadratic term can only use degree-one summands of the substituted variables.

\section{Infinite bases and nonfinite basability}
All bases and relative varieties in this paper are considered within the variety of commutative ai-semirings. Thus associativity of both operations, commutativity and idempotence of addition, commutativity of multiplication, and the two distributive laws are understood. Let $\Sigma_0$ consist of
\[
x\preceq y^2,\qquad x^2\approx yzt,
\]
and, for every $n\ge1$, the odd-cycle identity
\[
\sigma_n:\quad x\preceq x_1x_2+x_2x_3+\cdots+x_{2n}x_{2n+1}+x_{2n+1}x_1. \label{eq:oddcycles}
\]
Put $C(x,y)=x+xy$ and $D(x,y)=x+y+xy$. We use the following additional identities:
\begin{align*}
\xi &: x^2\approx D(x,y),\\
\alpha &: x_1x_4\preceq x_1+x_1x_2+x_2x_3+x_3x_4,\\
\beta &: xz\preceq x+xy+t+tz,\\
\eta &: z\preceq x+xy+yz,\\
\theta_n &: x_1x_{2n+2}\preceq x_1+x_{2n+2}+x_1x_2+x_2x_3+\cdots+x_{2n+1}x_{2n+2}.
\end{align*}

\begin{theorem}[explicit bases]\label{thm:bases}
Relative to the commutative ai-semiring axioms, complete bases are
\[
\begin{array}{c|c}
R_{01}&\Sigma_0\cup\{\xi,\alpha,\beta\}\\
R_{02}&\Sigma_0\cup\{\xi,\eta\}\\
R_{11}&\Sigma_0\cup\{\alpha\}\\
R_{12}&\Sigma_0\cup\{\theta_n:n\ge1\}.
\end{array}
\]
The claim of nonfinite basability will be proved separately below.
\end{theorem}

\begin{proof}
The first identity of $\Sigma_0$ makes all squares equal and greatest in the additive order; write their common value as $0$. The second identifies every cubic monomial with $0$, and multiplication by another variable preserves $0$. The identities $\sigma_n$ make every odd-cycle term zero. These observations account for all basic-zero cases.

The law $\alpha$ shortens a path from a marked endpoint by two edges: apply it to its first three edges, keep the other summands as additive context, and iterate. In $R_{01}$, an odd path between marks therefore gives a double-marked edge, which $\xi$ makes zero. For a cross-component query allowed by (a), first use $\alpha$ to connect a mark in the second component to the queried vertex and then apply $\beta$ to this marked edge and an edge at the first mark. In $R_{02}$, successive instances of $\eta$ mark the vertices at even distance from a mark. If two marks are at odd distance, this produces adjacent marks and $\xi$ applies. In $R_{11}$ the same path-shortening argument gives exactly (c). In $R_{12}$ the odd path in (d) gives a direct instance of $\theta_n$, unless the edge is already present. Every listed axiom is valid by Theorem~\ref{thm:criterion}. Since all monomial absorptions in that theorem have now been derived, decomposing an arbitrary identity into its monomial absorptions proves completeness.
\end{proof}

\begin{lemma}[rigidity of unmarked bipartite terms]\label{lem:rigidity}
Suppose that a term $s$ has only degree-one and square-free degree-two monomials, its quadratic graph is bipartite, and no variable occurs both as a degree-one and degree-two monomial. If $s\approx t$ holds in one of $R_{01},R_{02},R_{11},R_{12}$, then $s$ and $t$ have the same monomials after idempotent-addition normalization.
\end{lemma}

\begin{proof}
Every monomial of $t$ is absorbed by $s$. The graph criterion and its purely linear case say that $s$ absorbs only its displayed monomials, so $t$ is a subpolynomial of $s$. In particular, $t$ is also unmarked and bipartite. Applying the same argument in the other direction gives equality of the monomial sets.
\end{proof}

\begin{theorem}[nonfinite basability]\label{thm:nfb}
Each of $R_{01},R_{02},R_{11},R_{12}$ is nonfinitely based.
\end{theorem}

\begin{proof}
Fix one of the four semirings and suppose that a finite identity basis $F$ exists. Let $m$ bound the number of variables in an identity in $F$, and choose an odd integer $L>\max\{m,3\}$. Let
\[
w=x_1x_2+x_2x_3+\cdots+x_{L-1}x_L+x_Lx_1
\]
be the quadratic term of an $L$-cycle, and let $z$ be new. The odd-cycle identity gives $w\approx w+z$. Work with polynomials modulo the commutative ai-semiring axioms. At the first step that changes the polynomial, a basis identity $s\approx t$ is used under a substitution $\varphi$ and additive and possibly multiplicative contexts. Its contribution before the step is $\varphi(s)$ or $P\varphi(s)$, with $P$ nonempty in the latter case; each monomial of that contribution occurs in $w$. Missing contexts mean absence of an operation, not a constant in the language.

Substitutions are nonerasing because the language has no constants. If $P$ is present, every monomial of $s$ must have degree one: a quadratic monomial would give total degree at least three. The rigidity lemma then makes $s\approx t$ a polynomial equality, so this step cannot change $w$.

Suppose that $P$ is absent. The original term $s$ contains no square or monomial of degree at least three, since no instance of such a monomial can be a square-free quadratic monomial of $w$. A variable cannot occur both linearly and quadratically in $s$: its linear occurrence would force every monomial of its image to have degree two, making its quadratic occurrence have degree at least three. If the quadratic graph of $s$ had an odd cycle, it would have one of length at most $m$. Every variable on that cycle has a substitution image consisting of linear summands. Choose one support vertex from each image. The substituted cycle would give a homomorphism into $C_L$, hence an odd closed walk in $C_L$ of length less than $L$. Such a walk does not exist: after cancellation of immediate reversals, a nonempty odd closed walk in a cycle must traverse the whole cycle. Thus $s$ is an unmarked bipartite term with disjoint isolated linear summands. The rigidity lemma again makes the step trivial. This contradiction rules out every finite basis.
\end{proof}

\section{Fixed representatives and complete normal forms}
\label{sec:normal}
Throughout the rest of the paper put $\mathcal V=\V(R_{12})$. The notation $\mathcal V[E]$ means the subvariety obtained by adjoining the identities in $E$. The symbols
\[
C(t,u)=t+tu,\qquad D(t,u)=t+u+tu,\qquad L(a,b,c,d)=ab+bc+cd
\]
are abbreviations for terms. A superscript or subscript $D$ will indicate a specified double-marked context, rather than a new operation.

Let $M(v)$ be the variables occurring as linear summands of a term $v$, let $G(v)$ be its quadratic graph, and put $W(v)=M(v)\cap V(G(v))$. A component containing an edge has type U if it has no marked vertex, type S if only one bipartition class is marked, and type D if both classes are marked. In an S-component we call the marked class E and the other class O. Linear summands outside the quadratic graph are retained separately. Write $d(v)$ when the graph has a D-component.

The normal form supplied by Theorem~\ref{thm:criterion}(d) has two steps. A basic-zero term becomes $0$; otherwise, we add all edges between oppositely placed marked vertices in the same component. No other edge is added. In particular, a D-component is not in general a complete bipartite graph. It does, however, contain a double-marked edge, so its term absorbs a copy of $D(t,u)$.

We shall use the fixed set $\F_{25}$ in Table~\ref{tab:F25}. The purpose of the present section is to decide identities for a specified presentation. The assertion that these presentations exhaust all subvarieties is a separate reduction theorem proved in Section~\ref{sec:reduction}.

{\small
\setlength{\tabcolsep}{7pt}
\renewcommand{\arraystretch}{1.22}
\begin{longtable}{c@{\quad}l}
\caption{The fixed representative identities $\F_{25}$.}\label{tab:F25}\\
\toprule
Name & Identity or absorption inequality\\
\midrule
\endfirsthead
\multicolumn{2}{c}{Table \thetable\ (continued)}\\
\toprule Name & Identity or absorption inequality\\ \midrule
\endhead
\bottomrule
\endfoot
$\rho$ & $x\preceq xy$\\
$\eta$ & $z\preceq x+xy+yz$\\
$\omega$ & $y\preceq x+xy$\\
$\zeta$ & $z\preceq C(x,y)+zt$\\
$\lambda$ & $z\preceq D(x,y)+xz$\\
$\eta_D$ & $z\preceq x+xy+yz+D(t,u)$\\
$\omega_D$ & $y\preceq C(x,y)+D(t,u)$\\
$\zeta_D$ & $z\preceq D(x,y)+zt$\\
$\mu$ & $ad\preceq b+L(a,b,c,d)$\\
$\nu$ & $ad\preceq C(t,u)+L(a,b,c,d)$\\
$\mu_2$ & $ad\preceq b+c+L(a,b,c,d)$\\
$\mu_D$ & $ad\preceq b+L(a,b,c,d)+D(t,u)$\\
$\nu_D$ & $ad\preceq L(a,b,c,d)+D(t,u)$\\
$\delta$ & $ad\preceq L(a,b,c,d)$\\
$\xi$ & $x^2\approx D(x,y)$\\
$\beta$ & $xz\preceq C(x,y)+C(t,z)$\\
$s$ & $x^2\approx C(x,y)$\\
$\chi$ & $x^2\approx xy$\\
$\tau$ & $x\approx y$\\
$\alpha$ & $ad\preceq a+L(a,b,c,d)$\\
$\alpha_D$ & $ad\preceq a+L(a,b,c,d)+D(t,u)$\\
$\varepsilon$ & $ad\preceq a+b+c+L(a,b,c,d)$\\
$\kappa$ & $ad\preceq a+b+L(a,b,c,d)$\\
$\iota_2$ & $x_0x_5\preceq x_1+\sum_{j=0}^{4}x_jx_{j+1}$\\
$\iota_2^D$ & $x_0x_5\preceq x_1+\sum_{j=0}^{4}x_jx_{j+1}+D(t,u)$\\
\end{longtable}
}

The square identities in the table may be read as $0\preceq v$, since the converse inequality already holds. The identity $\tau$ may be read as $y\preceq x$. All variables displayed in an independent context are distinct from the path variables, but substitutions are permitted to identify them.

For $n\ge1$ define
\begin{align}
\gamma_n:\quad x_0y&\preceq
x_0+x_{2n}+\sum_{j=0}^{2n-1}x_jx_{j+1}+x_{2n}y,\label{eq:gamma}\\
\gamma_n^D:\quad x_0y&\preceq
x_0+x_{2n}+\sum_{j=0}^{2n-1}x_jx_{j+1}+x_{2n}y+D(t,u).\label{eq:gammaD}
\end{align}
For $H\subseteq\F_{25}$ and $0\le p\le q\le\infty$ put
\begin{equation}
\mathcal U(H,p,q)=
\mathcal V[H,\gamma_i:1\le i\le p,\gamma_j^D:1\le j\le q].
\label{eq:presentation}
\end{equation}
The threshold $0$ specifies an empty family and $\infty$ specifies the whole family.

\subsection{The closure of a term}
For every presentation \eqref{eq:presentation}, define the closure of a term as follows; no logical closure of $H$ is assumed. If $\tau\in H$, all terms have the same normal form. Otherwise, take the least extension of the environment normal form of $v$ that is closed under the following rules.
\begin{enumerate}[label=(\arabic*),leftmargin=2.2em]
\item For each $f=(q_f\preceq v_f)\in H$, for every map of the variables of $v_f$ into the variables of the term, allowing identifications, include the image of $q_f$ whenever all image summands of $v_f$ are present. A square query is interpreted as $0$, and its insertion sends the result to $0$.
\item Set $k=p$ if there is no D-component and $k=q$ otherwise. For each two marked vertices $a,b$ in the same bipartition class with distance at most $2k$, join $a$ to every neighbour of $b$. For $k=\infty$, require only that $a,b$ lie in the same component and class.
\item Reapply the environment rule. A square or odd cycle gives $0$. Otherwise, add all edges between oppositely placed marks in the same component.
\end{enumerate}
The condition $d(v)$ is evaluated on the enlarged term at each application of a rule. No rule introduces a variable. On $n$ variables there are at most $n+\binom n2$ eligible linear and square-free quadratic monomials. Consequently the least closed extension is obtained after finitely many strict enlargements, even when a threshold is infinite. Denote it by $C_{H,p,q}(v)$.

The second step expresses every vertex instance of the path families. An even walk may be shortened to a path; a shorter positive even path may be padded by an out-and-back step to any larger prescribed even length. If the two marked vertices coincide, the queried edge is already present. Thus the distance version gives precisely the same saturation as all permitted path instances.

\begin{theorem}[complete normal forms]\label{thm:normal}
For arbitrary nonempty terms $u,v$,
\[
\mathcal U(H,p,q)\models u\approx v
\quad\Longleftrightarrow\quad
C_{H,p,q}(u)=C_{H,p,q}(v).
\]
For a monomial $w$, this is equivalently
\[
\mathcal U(H,p,q)\models w\preceq v
\quad\Longleftrightarrow\quad
C_{H,p,q}(v)=0\ \text{or}\ w\in C_{H,p,q}(v).
\]
\end{theorem}

\begin{proof}
If $\tau\in H$, the presentation defines the trivial variety and the assertion holds with a single normal form. Assume henceforth that $\tau\notin H$. Each insertion is an instance of a defining identity with the remaining summands added as context. Hence $v\approx C_{H,p,q}(v)$ is valid. It remains to prove that polynomial substitutions yield no further identifications.

Apart from $\tau$, each variable in a fixed right-hand representative occurs in a quadratic summand. The same is true of all path representatives. If the image of such a variable has a monomial of degree at least two, one of those products has degree at least three, and the whole right side is $0$. Therefore a nonzero right side requires every variable image to be a nonempty sum of linear variables. Given any monomial of the substituted query, select the required vertices from these supports and select arbitrary support vertices for the remaining variables. All right-hand image monomials of this vertex substitution are in the polynomial expansion. The corresponding saturation step inserts the chosen query monomial. For a square query, choosing the same support vertex twice inserts a square. The argument also proves validity of all $\theta_n$ instances; an odd closed walk contains an odd cycle, giving all $\sigma_n$ instances.

To turn this observation into a separating algebra, fix a nonempty finite set $X$ containing the variables in $u,v$. Let $B_X(H,p,q)$ consist of all nonzero closed, nonempty subsets of the linear and square-free quadratic monomials on $X$, together with one element $0$. The latter represents the common square and is greatest for addition and absorbing for multiplication. Define
\[
A+B=C_{H,p,q}(A\cup B),\qquad
A B=C_{H,p,q}(A\cdot B).
\]
Here $A\cdot B$ denotes the ordinary polynomial product, expanded as a set of commutative monomials. The closure is extensive, monotone and idempotent (with $0$ as top), so addition is a semilattice operation. A nontrivial insertion can only occur in a term with a quadratic monomial. Its product with any nonempty term already contains a cubic monomial and is $0$, before and after insertion. Purely linear terms are unchanged. Consequently
\[
C_{H,p,q}(C_{H,p,q}(A)B)=C_{H,p,q}(AB),
\]
and likewise in the other factor. Associativity and distributivity follow from the polynomial operations. Every square of a nonempty polynomial contains a square or a monomial of degree at least three, and every product of three nonempty polynomials has degree at least three. Thus $A^2=ABC=0$ in this algebra. Since $0$ is the additive maximum, the two basic identities $x\preceq y^2$ and $x^2\approx yzt$ hold. The support argument proves all remaining defining identities, so the algebra belongs to $\mathcal U(H,p,q)$.

Assigning each variable to its singleton linear term evaluates every term to its closure. Distinct closures thus refute the queried identity. In particular, a missing query is separated by an explicitly defined finite algebra belonging to the specified variety. This algebra has at most $2^{|X|+\binom{|X|}{2}}$ elements: there are at most $2^N-1$ nonempty subsets of the $N=|X|+\binom{|X|}{2}$ monomials, and one additional zero.
\end{proof}

\begin{remark}\label{rem:context}
If the original right side absorbs $D_0=t_0+u_0+t_0u_0$, every local use of an unconditioned rule can be replaced by the corresponding D-rule with this same $D_0$. All edge insertions used below preserve that context. Thus a conditional graph derivation is valid without assuming that its unconditioned identity holds throughout the variety. This observation will be used explicitly in the reduction proof.
\end{remark}
\section{Reduction of arbitrary identities}
\label{sec:reduction}

We now prove the assertion needed to pass from specified presentations to all subvarieties. All equivalences in this section are relative to $\mathcal V=\V(R_{12})$, unless another environment is explicitly stated. In particular, equivalence means equality of the defined subvarieties.

\begin{theorem}[global reduction]\label{thm:reduction}
For every identity $u\approx v$ there is a finite set
\[
E\subseteq\F_{25}\cup\{\gamma_n,\gamma_n^D:n\ge1\}
\]
such that $\mathcal V[u\approx v]=\mathcal V[E]$.
\end{theorem}

We prove the theorem by giving a complete reduction of a monomial absorption $q\preceq v$. First replace $v$ by its environment normal form. An absorption with zero right side or an already present query needs no representative. Henceforth the right side is nonzero and the query is absent. When the graph has an edge, put
\[
Z(v)=\begin{cases}
\xi,&d(v),\\
s,&\neg d(v),\ W(v)\ne\varnothing,\\
\chi,&W(v)=\varnothing.
\end{cases}
\]
The labels U, S, D and the oriented parts E, O have the meanings fixed in Section~\ref{sec:normal}. A graph neighbourhood never includes isolated linear summands.

\subsection{Elementary foldings}

\begin{lemma}\label{lem:elementaryfolds}
The reductions in Table~\ref{tab:elementaryfolds} hold. In the last row, the only excluded cross-component case is a query from the E part of one S-component to the O part of another S-component when $d(v)$ is false.
\end{lemma}

\begin{table}[ht]
\centering\small
\caption{Linear, unmarked-component and zero reductions.}
\label{tab:elementaryfolds}
\begin{tabular}{@{}p{0.56\textwidth}p{0.24\textwidth}p{0.10\textwidth}@{}}
\toprule
Absent query & $\neg d(v)$ & $d(v)$\\
\midrule
Linear variable in the E part of an S-component & $\eta$ & $\eta_D$\\
Linear variable in its O part & $\omega$ & $\omega_D$\\
Linear variable in a D-component & --- & $\lambda$\\
Linear variable in a U-component & $\rho$ if $W=\varnothing$; $\zeta$ otherwise & $\zeta_D$\\
Missing opposite-part edge in a U-component & $\delta$ if $W=\varnothing$; $\nu$ otherwise & $\nu_D$\\
New graph-external linear variable; square; higher monomial; edge involving a graph-external variable; same-part edge; remaining cross-component edge & $Z(v)$ & $Z(v)$\\
\bottomrule
\end{tabular}
\end{table}

\begin{proof}
For a new linear query $z$ in the E part, map $z$ to itself, all other E vertices of its component to $x$, and all O vertices to $y$. The image right side is contained in $x+xy+yz$. Other S-components fold into the marked edge $x+xy$, and U-components into $xy$. Isolated linear variables may be replaced by $xy$. Thus the original absorption implies $\eta$. If D-components are present, map all of them to a separate $D(t,u)$, giving $\eta_D$. Conversely, repeated two-edge propagation of a mark gives any vertex in the E part, retaining the same D-context when necessary. For an O-part linear query, folding the two parts to $x,y$ gives $\omega$ or $\omega_D$; successive applications propagate marks along a path.

For a query $z$ in a D-component, map $z$ to $z$, the other vertices in its part to $y$, and the opposite part to $x$. All terms map into $x+y+xy+xz$, so the original law implies $\lambda$. Conversely, an adjacent marked pair exists by the environment rule. The law $\lambda$ marks every neighbour of a marked vertex having a marked neighbour, and connectedness propagates this to the whole component.

A U-component folds to an ordinary edge. Keeping a queried linear endpoint gives $\rho$; an additional marked edge or double-marked edge gives $\zeta$ or $\zeta_D$. Conversely, apply the appropriate representative to an edge incident with the query and to the indicated context. For a missing opposite-part edge $ad$ in a U-component, preserve its endpoints and send the other vertices in their two parts to $c,b$, respectively. No original edge maps to $ad$, so all edges map into $ab+bc+cd$. Fold other components to an ordinary, marked or double-marked edge according to the global context. This extracts $\delta,\nu$ or $\nu_D$. In the reverse direction, repeatedly shorten an odd path between the query endpoints by two edges, retaining the chosen context.

For the last row, orient each U-component freely. When there is no D-component, send all marked parts of S-components to $x$ and their opposite parts to $y$. Except for the excluded S--S E/O query, the two endpoints of a same-part or cross-component query can be sent to the same variable. With a D-component, either orientation is permitted and all components fold into $D(x,y)$. A fresh graph-external query variable may be replaced by a square; an isolated linear variable occurring in a quadratic query may be replaced by $xy$, making that query cubic while keeping the right side unchanged after normalization. Squares and higher queries already equal $0$. We obtain $0\preceq xy$, $0\preceq C(x,y)$ or $0\preceq D(x,y)$, namely $\chi,s$ or $\xi$. Conversely, the appropriate representative makes an edge, a marked edge or a double-marked edge already absorbed by $v$ equal to $0$, and therefore makes $v=0$.

If the graph has no edges, a missing linear query is fresh, and collapsing the displayed variables gives $y\preceq x$. For a query of degree at least two, collapse all variables to $x$ to obtain $0\preceq x$. Either law forces the trivial variety, so the representative is $\tau$.
\end{proof}

\subsection{Two descriptions of path closure}

For a positive integer $k$ and an S-component, join two marks in an auxiliary graph when their distance in the original graph is at most $2k$. Call the auxiliary components the $k$-groups. At $k=\infty$, all marks of the component form one group. For a group $J$ put
\[
N_J=N_G(J),\qquad L_J=N_G(N_J).
\]

\begin{lemma}[group closure]\label{lem:groups}
On an S-component, adjoining $\gamma_k$ to the environment shares the original neighbours of marks in each $k$-group: the added edges are $J\times N_J$. Adjoining $\mu$ as well, with $k\ge1$, adds exactly $L_J\times N_J$ for every $k$-group. No linear term, component or bipartition changes. Under $\mu$, each D-component becomes complete bipartite, and U-components remain unchanged. The conditional versions have the same action on a term containing a fixed D-context.
\end{lemma}

\begin{proof}
The distance formulation in Section~\ref{sec:normal} first shares neighbours between marks at distance at most $2k$, and transitivity shares them within a group. A new short path between different groups would have a subpath between marks from different groups with no internal mark. Only its end edges could be new. Replace the end marks by original marks in their respective groups that supplied those edges. This produces an original path of the same length, contradicting the definition of the groups. Thus sharing creates no further groups and is saturated.

The law $\mu$ implies $\gamma_1$: on the path $a-b-c-d$ with $a,c$ marked, apply $\mu$ to the reversed walk $d-c-b-a$. After a group has shared neighbours, if $z\in N_J$ and $x\in L_J$, choose $w\in J$ and $r\in N_J$ with $rx$ an original edge. The walk $z-w-r-x$ has internal mark $w$, so $\mu$ adds $zx$.

For saturation, different $N_J$ are disjoint, since a common neighbour would join their marks at distance two. A marked vertex belongs to $L_J$ only if it is in $J$. After all rectangles $L_J\times N_J$ are filled, the neighbours of marks of $J$ are still exactly $N_J$, and the neighbours of $N_J$ are still exactly $L_J$: every edge newly incident with $N_J$ has its other endpoint in $L_J$, while edges added for $K\ne J$ have their O endpoint in the disjoint set $N_K$. Thus $\mu$ adds nothing further.

Suppose the new graph joined different mark groups by a path of length at most $2k$. Along its O vertices, choose two successive occurrences of distinct neighbour sets $N_J,N_K$, with no intervening O vertex in any $N_T$. The intervening path has only original interior edges; replace its two end O vertices by the original neighbours that supplied its end edges. Attach an original mark at either end. If the intervening O-to-O segment has length $r$ and the assumed marked path has length $\ell$, then $r+2\le\ell$. The replacements and the two original marks give a walk in the original graph of length $r+2$. Hence $d_G(J,K)\le r+2\le\ell\le2k$, contradicting that $J,K$ are distinct $k$-groups. When $k=\infty$, there is only one mark group in the component, so this case does not arise. This proves saturation under the threshold rules as well.

In a D-component the environment connects all opposite-part marks. The law $\mu$ implies $\kappa$ by adding the first mark as context. Fix opposite marks $a,b$. Along a path, use the walks $b-a-v_1-v_2$, $a-b-v_2-v_3$, and so on; $\kappa$ alternately connects $b$ and $a$ to every vertex in the appropriate part. A further application of $\mu$ fills all missing opposite-part edges. The support argument and separating algebra of Theorem~\ref{thm:normal} now prove that the saturated descriptions are complete. A fixed D-context can be kept at each step as in Remark~\ref{rem:context}.
\end{proof}

\subsection{A marked query endpoint}

Suppose that $xy$ is missing, $x$ is marked, and $x,y$ lie in opposite parts of one component. The case where both endpoints are marked is already an environment consequence. If $y$ has a marked neighbour, let $h$ be the least positive threshold for which the group containing $x$ also contains a marked neighbour of $y$.

\begin{lemma}\label{lem:markedfold}
The reductions in Table~\ref{tab:markedfold} are equivalences.
\end{lemma}

\begin{table}[ht]
\centering\small
\caption{A marked endpoint $x$ and an unmarked endpoint $y$.}
\label{tab:markedfold}
\begin{tabular}{@{}lcc@{}}
\toprule
Query component and global context & $y$ has a marked neighbour & $y$ has none\\
\midrule
S-component, $\neg d(v)$ & $\gamma_h$ & $\alpha$\\
S-component, $d(v)$ & $\gamma_h^D$ & $\alpha_D$\\
D-component & $\varepsilon$ & $\kappa$\\
\bottomrule
\end{tabular}
\end{table}

\begin{proof}
In a D-component, if $y$ has marked neighbour $c$, choose an opposite mark $b$. The environment supplies $xb,bc$, so the marked walk $x-b-c-y$ gives the query by $\varepsilon$. Conversely, preserve $x,y$ as $a,d$, send the remaining vertices in their parts to $c,b$, and fold other components into a double-marked edge. This maps the right side into the right side of $\varepsilon$. If $y$ has no marked neighbour, instead send $x$ and every mark in its part to $a$, the other vertices of that part to $c$, $y$ to $d$, and the other part to $b$. No actual edge maps to $ad$. The resulting law is $\kappa$, whose reverse implication is the alternating path argument in Lemma~\ref{lem:groups}.

For an S-component with no marked neighbour at $y$, the latter folding has marks only at $a$, giving $\alpha$, or $\alpha_D$ if other D-components are folded to a separate $D(t,u)$. Repeated marked-end path shortening gives the converse.

Now let the query have finite threshold $h$. Sharing neighbours within $h$-groups proves that $\gamma_h$ implies the query. Conversely, the first folding above, now with marks only at $a,c$, extracts $\gamma_1$. For $h\ge2$, let $J$ be the $(h-1)$-group of $x$. Every outside mark has distance at least $2h$ from $J$, and $d_G(J,y)\ge2h-1$. Define on this component
\begin{equation}\label{eq:distancefold}
f(z)=\min\{d_G(z,J),\ 2h+(d_G(z,J)\bmod2)\}.
\end{equation}
Adjacent vertices have distances differing by one, because $J$ lies in a single bipartition class. Their images remain adjacent on the path $0,1,\ldots,2h+1$. Marks in $J$ map to $0$, and other marks to $2h$. The query maps to $(0,2h-1)$ or $(0,2h+1)$. Add any missing path summands as context. In the second case we have $\gamma_h$ directly. In the first, apply the already extracted $\gamma_1$ to the walk $0-(2h-1)-(2h)-(2h+1)$ to obtain the same result. Other S/U components fold into a marked or ordinary edge and isolated linear summands into an existing product. If D-components are present, retain their common independent D-edge in both extractions and in every reverse step. This gives exactly $\gamma_h^D$.
\end{proof}

\subsection{Unmarked query endpoints}

The next step must precede the use of the $\mu$-closure. It is not legitimate to assume $\mu$ merely because a query has unmarked endpoints.

\begin{lemma}\label{lem:extractmu}
Suppose two unmarked query endpoints are in opposite parts of one S-component, or in the E and O parts of different S-components when $\neg d(v)$. The query implies $\mu$ if $\neg d(v)$ and implies $\mu_D$ in the former case if $d(v)$. In a D-component, a missing edge with two unmarked opposite-part endpoints is equivalent to $\mu_2$ if both endpoints have marked neighbours, and to $\{\mu_2,\kappa\}$ otherwise.
\end{lemma}

\begin{proof}
Orient the query as $x\in\mathrm E,y\in\mathrm O$. Send $x$ to $a$, $y$ to $d$, all other E vertices to $c$, and all other O vertices to $b$. Orient all S-components by their marked parts. The query is missing and both endpoints are unmarked, so edges map into $ab+bc+cd$ and marks into $c$. Isolated linear summands map to $bc$. The resulting absorption
\[
ad\preceq c+ab+bc+cd
\]
is $\mu$ after reversing the path. In the conditional case fold all D-components into a separate $D(t,u)$, giving $\mu_D$. Choose an actual double-marked edge $D_0\preceq v$. Every subsequent local $\mu$ insertion is then valid using $\mu_D$ with this $D_0$, even though the unconditioned law has not been deduced.

For a D-component with unmarked endpoints $x\in L,y\in R$, send them to $a,d$ and the remaining vertices of $L,R$ to $c,b$. All marks map to $b,c$, giving $\mu_2$. Let $N_L=N_G(W\cap R)$ and $N_R=N_G(W\cap L)$. If $x\in N_L$ and $y\in N_R$, select marked neighbours $b,c$ of $x,y$. The edge $bc$ is an environment consequence, so $\mu_2$ applies to $x-b-c-y$. In fact it fills exactly $N_L\times N_R$, which does not enlarge these two neighbour sets.

Otherwise, interchange the parts if necessary so that $y$ has no marked neighbour. Send $x$ and all marks in $L$ to $a$, the other L vertices to $c$, $y$ to $d$, and the remaining R vertices to $b$. Neither $x$ nor a marked L vertex is adjacent to $y$, so no edge maps to $ad$; the marks map into $a,b$. This extracts $\kappa$. Conversely, $\kappa$ connects each mark to its entire opposite part; then $\mu_2$ fills the whole component. Other components fold into the available double-marked edge throughout.
\end{proof}

Let $v_{\mu}$ be the closure of a right side under $\mu$ and the environment. When only $\mu_D$ has been extracted, the same term is valid as a replacement because $D_0\preceq v$ remains present. If the query is already present in $v_\mu$, the extracted law alone completes the reduction. Otherwise, in an S-component put $N=N_G(W)$ and $L_2=N_G(N)$, now using this closed graph, and orient the remaining missing query as $x\in\mathrm E,y\in\mathrm O$.

\begin{lemma}[remaining S-component queries]\label{lem:internalfold}
In the environment with $\mu$, the three reductions are
\[
\begin{array}{c|c}
\text{condition}&\text{representative}\\\hline
y\in N,\ x\in L_2&\gamma_h\\
y\in N,\ x\notin L_2&\iota_2\\
y\notin N&\alpha,
\end{array}
\]
where $h\ge2$ is the first threshold in Lemma~\ref{lem:groups} that inserts the query. If a fixed D-context is present, all representatives and all local uses of $\mu$ may be replaced by their D-versions.
\end{lemma}

\begin{proof}
For the middle row, use the folding
\[
W\mapsto a,\quad N\mapsto b,\quad
\mathrm E\setminus(W\cup\{x\})\mapsto c,\quad
\mathrm O\setminus N\mapsto d,\quad x\mapsto e.
\]
There is no edge from $W$ to $\mathrm O\setminus N$ or from $x$ to $N$. Thus the query implies
\begin{equation}\label{eq:psi}
\psi:\quad be\preceq a+ab+bc+cd+de.
\end{equation}
The substitution $(x_0,x_1,x_2,x_3,x_4,x_5)=(b,a,b,c,d,e)$ shows that $\iota_2$ implies $\psi$. Conversely, apply $\psi$ to the subpath $x_1,x_2,x_3,x_4,x_5$ to obtain $x_2x_5$, and then $\mu$ to $x_0-x_1-x_2-x_5$. Notice also that $\iota_2$ implies $\mu$ by the substitution $(a,b,a,b,c,d)$.

For the reverse implication, a mark $w$ and one of its neighbours $z$ satisfy $zv_2\preceq v$ on every two-edge walk $w-v_1-v_2$, by $\mu$. Once $zv_{2r}$ has been inserted, $\psi$ on $w-z-v_{2r}-v_{2r+1}-v_{2r+2}$ inserts $zv_{2r+2}$. Hence $\iota_2$ joins every marked neighbour to every vertex in the marked part of the component, proving the middle row. If $y\notin N$, map $x$ and all marks to $a$, other E vertices to $c$, $y$ to $d$, and other O vertices to $b$. This gives $\alpha$. The laws $\alpha,\mu$ complete every marked component: first connect marks to their entire opposite part by $\alpha$, and then apply $\mu$ through an internal mark. This proves the last row.

For the first row, Lemma~\ref{lem:groups} gives a finite first threshold $h$, since at infinite threshold the rectangle is $L_2\times N$. The converse from $\gamma_h$ is immediate from that lemma. We give the extraction, including the case where $x$ is unmarked.

At threshold one, let $\mathcal J_x$ be the groups $J$ for which $x\in L_J$. This collection is nonempty. The vertex $y$ belongs to the neighbour set of one other group, since the query is missing. Map $x$ and all marks in $\mathcal J_x$ to $0$, other marks to $4$, and other E vertices to $2$. Map O vertices neighbouring an ``other'' mark to $3$ and the remaining O vertices to $1$. No edge maps to $0-3$: such an edge at $x$ would put the supplying mark group in $\mathcal J_x$, and such an edge at a mark would merge the two groups at threshold one. No edge maps to $4-1$ either. Thus all edges map into $0-1-2-3-4$ and the query maps to $0-3$. After adding the last edge $4-5$, $\mu$ on $5-4-3-0$ gives $0-5$, extracting $\gamma_2$.

If $h\ge3$, all marks at distance at most two from $x$ lie in one $(h-1)$-group $J$, because their mutual distances are at most four. We have $d_G(J,x)\in\{0,2\}$ and $d_G(J,y)\ge2h-1$. The distance truncation \eqref{eq:distancefold} maps the query to one of
\[
(0,2h-1),\quad(0,2h+1),\quad(2,2h-1),\quad(2,2h+1).
\]
All are absent from the target path. For a query $(2,r)$, where $r\in\{2h-1,2h+1\}$, use the substitution
$(x_0,x_1,x_2,x_3,x_4,y)=(0,1,2,r,2h,r)$ in the already extracted
$\gamma_2$. The required walk is $0-1-2-r-(2h)-r$; its marks are
$0,2h$, and its edge $2r$ is the inserted query. The result is $0r$,
so the left endpoint has moved to $0$. A query ending at $2h-1$ then gives the final query by $\mu$ on $(2h+1)-(2h)-(2h-1)-0$. Thus all four possibilities yield $\gamma_h$.

Other S/U components fold to existing marked or ordinary edges; isolated linear summands fold to an existing product. In every conditional extraction, all D-components map to the same separate $D(t,u)$. In every reverse insertion, retain the actual $D_0$ from the original term. This proves the conditional statement using $\mu_D$, not an assumed global $\mu$.
\end{proof}

\subsection{Cross-component queries and completion of the reduction}

It remains to handle two S-components $C_1,C_2$ when $\neg d(v)$, with query $x\in\mathrm E(C_1)$, $y\in\mathrm O(C_2)$. Write $N_i=N_G(W\cap C_i)$ and $L_i=N_G(N_i)$.

\begin{lemma}\label{lem:crossfold}
If $x$ is marked, the query is equivalent to $\beta$ when $y\in N_2$ and to $\{\beta,\alpha\}$ otherwise. If $x$ is unmarked, first extract $\mu$ and replace $v$ by $v_\mu$. The remaining query, relative to $\mu$, is equivalent to
\[
\begin{cases}
\{\beta\},&y\in N_2,\ x\in L_1,\\
\{\beta,\iota_2\},&y\in N_2,\ x\notin L_1,\\
\{\beta,\alpha\},&y\notin N_2.
\end{cases}
\]
\end{lemma}

\begin{proof}
Fold the two query components to two independent marked edges, taking $x$ to the marked endpoint of the first and $y$ to the opposite endpoint of the second. This extracts $\beta$ in every case, even if $x$ was not marked. If $x$ is marked and $y$ has a marked neighbour, apply $\beta$ directly to an edge at $x$ and this marked edge at $y$. If $y$ has no marked neighbour, map $x$ and all E marks to $a$, the other E vertices to $c$, $y$ to $d$, and other O vertices to $b$. There is no edge $xy$ between the components and no marked neighbour at $y$, so this extracts $\alpha$. Conversely, use $\alpha$ to connect a mark in $C_2$ to $y$, followed by $\beta$.

When $x$ is unmarked, Lemma~\ref{lem:extractmu} justifies the preliminary $\mu$-closure. If $x\notin L_1$, map the marks in $C_1$ to $0$, $N_1$ to $1$, the other E vertices except $x$ to $2$, the other O vertices to $3$, and $x$ to $4$. Fold $C_2$ and the other components into the marked edge $0-1$. The missing query maps to $1-4$ and the right side into the four-edge path with mark $0$. This is \eqref{eq:psi}, and hence gives $\iota_2$. When $y\notin N_2$, the preceding four-vertex folding gives $\alpha$.

For the converse in the first unmarked case, choose an edge $xr$ with $r\in N_1$ and a mark $w$ adjacent to $r$. The law $\beta$ adds $wy$, and $\mu$ on $y-w-r-x$ adds $xy$. In the second case, $\beta$ first makes $y$ a neighbour of a mark in $C_1$, and the even-walk propagation of $\iota_2$ from Lemma~\ref{lem:internalfold} reaches $x$. In the third case, $\alpha,\mu$ first complete both marked components, reducing to the first case. This proves both directions of every reduction.
\end{proof}

\begin{proof}[Proof of Theorem~\ref{thm:reduction}]
Decompose an identity into finitely many monomial absorptions and use the environment normalization. Purely linear right sides, linear queries, graph-external queries, zero queries, U-component queries and all but one kind of cross-component query are covered by Lemma~\ref{lem:elementaryfolds}. For a same-component opposite-part edge in a marked component, two marked endpoints are already handled by the environment, one marked endpoint by Lemma~\ref{lem:markedfold}, and two unmarked endpoints by Lemmas~\ref{lem:extractmu} and \ref{lem:internalfold}. The exceptional S--S cross-component queries are covered by Lemma~\ref{lem:crossfold}. These cases exhaust all monomials, component positions and marking patterns. Each reduction is a two-way relative implication and produces finitely many representatives. Taking their union proves the theorem.
\end{proof}

\begin{corollary}[existence of presentations]\label{cor:presentation}
Every subvariety of $\mathcal V$ has a presentation \eqref{eq:presentation}.
\end{corollary}

\begin{proof}
Apply Theorem~\ref{thm:reduction} to each identity in any defining set and take the union of its representatives. The fixed part is a subset of the finite set $\F_{25}$. Folding back two edges gives $\gamma_{n+1}\Rightarrow\gamma_n$ and $\gamma_{n+1}^D\Rightarrow\gamma_n^D$, while adding context gives $\gamma_n\Rightarrow\gamma_n^D$. Thus each index set can be replaced by its downward closure: it is empty, a finite initial segment, or all positive integers. These are encoded by $0$, a positive integer, or $\infty$, with $p\le q$.
\end{proof}
\section{Canonical signatures and the full lattice}
\label{sec:canonical}

We first determine the exact path thresholds of an arbitrary presentation. This includes a converse proof on paths of unrestricted length.

\subsection{Exact thresholds}

Let
\begin{align*}
I_0&=\{\rho,\omega,\zeta,\delta,\nu,\alpha,\beta,s,\chi,\tau,\iota_2\},\\
I_D&=I_0\cup\{\omega_D,\zeta_D,\nu_D,\alpha_D,\xi,\iota_2^D\}.
\end{align*}
For $H\subseteq\F_{25}$ and $0\le p\le q\le\infty$ put
\[
a=\max\{p,\mathbf1_{\mu\in H}\},\qquad
b=\max\{q,\mathbf1_{H\cap\{\mu,\mu_D\}\ne\varnothing}\}.
\]

\begin{theorem}[exact thresholds]\label{thm:thresholds}
The true unconditioned and conditioned path indices in $\mathcal U(H,p,q)$ have respective suprema
\begin{align}
p^*&=\begin{cases}
\infty,&H\cap I_0\ne\varnothing\text{ or }(\eta\in H\text{ and }a\ge1),\\
a,&\text{otherwise},
\end{cases}\label{eq:pthreshold}\\
q^*&=\begin{cases}
\infty,&H\cap I_D\ne\varnothing\text{ or }(H\cap\{\eta,\eta_D\}\ne\varnothing\text{ and }b\ge1),\\
b,&\text{otherwise}.
\end{cases}\label{eq:qthreshold}
\end{align}
The supremum of the empty index set is understood to be $0$.
\end{theorem}

\begin{proof}
For the unconditioned family, consider the right side of $\gamma_n$: a path $0-1-\cdots-(2n)-(2n+1)$ with marks $0,2n$ and query $0\,(2n+1)$. It is an S-component. The input path laws already give every index at most $p$, and $\mu$ implies $\gamma_1$ by Lemma~\ref{lem:groups}, giving the lower bound $a$.

Each member of $I_0$ gives every path index. The laws $\rho,\omega,\zeta$ eventually mark the whole path, after which the environment joins the two opposite-part queried vertices. The laws $\delta,\nu$ repeatedly shorten a three-edge subpath; the marked-edge context for $\nu$ is already present. The law $\alpha$ does the same starting at the marked endpoint. The law $\beta$ connects that mark directly to a neighbour of the other mark. Each of $s,\chi,\tau$ makes the right side zero. Finally, $\iota_2$ connects a marked neighbour to the entire marked part by the even-walk induction in Lemma~\ref{lem:internalfold}. This proves the first infinity condition. If $\eta$ and one path law are available, $\eta$ marks the whole even part, and repeated uses of $\gamma_1$ along length-two paths share their neighbours. Thus the second infinity condition is also sufficient.

To prove necessity, assume neither condition holds and $p<\infty$. We show that the closure of every query with $n>a$ omits its queried edge. On the S-path, any remaining fixed rule other than possibly $\eta,\mu$ requires a D-context or adjacent opposite-part marks: these are the conditional representatives and $\lambda,\mu_2,\kappa,\varepsilon,\xi$. The rules with a global marked-edge context and the other zero rules are precisely those already excluded by $I_0$.

If $\eta\in H$, necessarily $p=0$ and $\mu\notin H$. Its iterations mark exactly the even part. No edge changes, no D-component is formed, and none of the excluded-pattern premises becomes applicable. The missing query remains missing for every $n\ge1$.

If $\eta\notin H$, the marks stay exactly $0,2n$. Without $\mu$ the graph stays unchanged for $n>p$, since distinct marks remain farther apart than $2p$. If $\mu$ is present and $n>\max\{p,1\}$, it can add only the edge $(2n-2)\,(2n+1)$. Indeed, a three-edge walk having internal mark $0$ returns through its unique neighbour and adds no new edge. At mark $2n$, its two neighbours $2n-1,2n+1$ produce exactly the displayed additional edge; afterwards they share the same two-step E-neighbourhood. No further $\mu$ edge is possible. This added edge does not shorten the distance $2n$ between the two marks, and so does not activate a path rule. The graph remains an S-component with no adjacent marks. Thus none of the other remaining fixed rules becomes applicable at any later stage. Theorem~\ref{thm:normal} turns this saturated missing query into a counterexample algebra. The upper bound is therefore exactly $a$.

For $\gamma_n^D$, retain a disjoint double-marked edge throughout. The effective path threshold is $q$, and $\mu$ or $\mu_D$ gives the first conditioned index. Each extra member of $I_D$ acts as its unconditioned version on the long path with this fixed context; $\xi$ makes the context itself zero. The same propagation and path-shortening arguments prove both infinity conditions in \eqref{eq:qthreshold}.

If no infinity condition occurs, the independent D-edge remains unchanged. The local rules $\lambda,\mu_2,\kappa,\varepsilon$ can act on that edge only trivially, and cannot affect the disconnected S-path. On the latter, the only possible actions are those of $\eta$ or $\eta_D$, of $\mu$ or $\mu_D$, and the threshold $q$. No O-part mark or connection to the D-edge is created. The preceding two invariant arguments therefore apply with $q,b$ in place of $p,a$. This proves the conditioned upper bound. Infinite input thresholds already give the entire corresponding family, completing all cases.
\end{proof}

\subsection{Canonical signatures}

For $f\in\F_{25}$ write $q_f\preceq v_f$ for its absorption form, with square queries interpreted as $0$. Define
\begin{align}
H^*&=\{f\in\F_{25}:C_{H,p,q}(v_f)=0\text{ or }q_f\in C_{H,p,q}(v_f)\},\label{eq:hclosure}\\
\Can(H,p,q)&=(H^*,p^*,q^*).\label{eq:canonical}
\end{align}
The convention for $\tau$ is that its query is a fresh linear variable. If $\tau\in H$, the result is $(\F_{25},\infty,\infty)$.

\begin{theorem}[canonical signatures]\label{thm:signatures}
The fixed points of $\Can$ are in bijection with $\Sub(\mathcal V)$ through \eqref{eq:presentation}. The operator is well defined and idempotent. The lattice is countably infinite.
\end{theorem}

\begin{proof}
Corollary~\ref{cor:presentation} supplies a presentation for every subvariety. Theorems~\ref{thm:normal} and \ref{thm:thresholds} show that $H^*,p^*,q^*$ are its exact representative truth coordinates. Adjoining true identities preserves the defined variety, so the normalized triple defines the same variety and is fixed by $\Can$. If two normalized triples differ, a fixed representative or a path identity at a finite index separates them. Conversely, equal triples give equal presentations. This proves both existence and uniqueness, rather than merely uniqueness among already specified presentations.

There are finitely many choices of $H$ and countably many pairs of thresholds. For the lower bound, Theorem~\ref{thm:thresholds} applied with $H=\varnothing$ gives exact thresholds $(n,n)$ for $\mathcal V[\gamma_n]$. Consequently
\[
\mathcal V[\gamma_1]\supsetneq\mathcal V[\gamma_2]\supsetneq\cdots.
\]
The missing query proof in that theorem and $B_X(\varnothing,n,n)$ give a finite separating algebra at every step, so all these containments are strict.
\end{proof}

\subsection{A finite implication calculus}

Put $g=\gamma_1$, $g_D=\gamma_1^D$, and
$B=\F_{25}\cup\{g,g_D\}$. These two additional names record whether a
path threshold is positive; they do not identify a finite threshold with
an infinite one. An implication between subsets of $B$ will always mean
an implication relative to $\mathcal V$.

\begin{lemma}\label{lem:finitecalculus}
The implications between the identities in $B$ are generated by
Table~\ref{tab:horn}. In particular, the table includes all implications
with arbitrarily many premises. On the right side of any identity in
$B$, the closure under $\mathcal U(H,p,q)$ depends on $p,q$ only through
whether each is zero or positive.
\end{lemma}

\begin{table}[ht]
\centering\small
\caption{A complete implication calculus on $B$. Each row is read from left to right.}
\label{tab:horn}
\begin{tabular}{@{}ll@{\qquad}ll@{}}
\toprule
Premise & Consequences & Premise & Consequences\\\midrule
$\alpha$ & $\alpha_D,g$ & $\alpha_D$ & $g_D,\kappa$\\
$\beta$ & $g,\xi$ & $\chi$ & $\beta,\rho$\\
$\delta$ & $\nu$ & $\eta$ & $\eta_D$\\
$\eta_D$ & $\lambda$ & $g$ & $g_D$\\
$g_D$ & $\varepsilon$ & $\kappa$ & $\varepsilon$\\
$\iota_2$ & $\iota_2^D,\mu$ & $\iota_2^D$ & $\mu_D$\\
$\lambda$ & $\kappa,\mu_2$ & $\mu$ & $g,\mu_D$\\
$\mu_2$ & $\varepsilon$ & $\mu_D$ & $g_D,\kappa,\mu_2$\\
$\nu$ & $\alpha,\iota_2,\nu_D$ & $\nu_D$ & $\alpha_D,\iota_2^D$\\
$\omega$ & $\alpha,\eta,\omega_D$ & $\omega_D$ & $\alpha_D,\eta_D$\\
$\rho$ & $\delta,\zeta$ & $s$ & $\beta,\omega$\\
$\tau$ & $\chi$ & $\xi$ & $\zeta_D$\\
$\zeta$ & $\omega,\zeta_D$ & $\zeta_D$ & $\nu_D,\omega_D$\\\midrule
$\alpha,\mu$ & $\iota_2$ & $\alpha_D,\mu_D$ & $\iota_2^D$\\
$\eta,g$ & $\alpha,\iota_2$ & $\eta_D,g_D$ & $\alpha_D,\iota_2^D$\\
$\beta,\rho$ & $\chi$ & $\omega,\xi$ & $s$\\
$\omega,\nu_D$ & $\nu$ & $\omega,\zeta_D$ & $\zeta$\\
\bottomrule
\end{tabular}
\end{table}

\begin{proof}
We first prove the positive implications. Adding an independent context
gives the D-versions. A context may also be identified with an edge
already on the path. Thus $\delta$ implies $\nu$, $\nu$ implies both
$\alpha$ and $\mu$, and $\mu_D$ implies $\kappa,\mu_2$. The analogous
statements hold in a fixed D-context. Lemma~\ref{lem:groups} gives
$\mu\Rightarrow g$, and the marked-endpoint argument gives
$\alpha\Rightarrow g$. Substituting $D(a,b)$ as the context on
$a-b-c-d$, with $a,c$ marked, gives $g_D\Rightarrow\varepsilon$.
The implications $\kappa,\mu_2\Rightarrow\varepsilon$ separately add
marks to their right sides. Repeated use of $\lambda$ marks a whole
D-component, after which the environment supplies the queries of
$\kappa$ and $\mu_2$.

The law $\omega$ marks successively every vertex of a marked component;
it therefore implies $\eta$ and $\alpha$. The law $\zeta$ marks every
vertex of every component in the presence of a marked edge; it implies
$\omega$ and its indicated conditional consequence. These arguments
with a retained D-edge prove the rows for $\omega_D,\zeta_D$.
The substitution of $D(x,y)$ into the conditional two-edge marking law
gives $\eta_D\Rightarrow\lambda$. The law $\rho$ marks both endpoints
of every edge; the environment then gives $\delta$, and an added
context gives $\zeta$. The zero laws imply the displayed consequences
because the requisite marked, double-marked or ordinary edge becomes
zero. Identifying the two queried vertices in $\beta$ gives $\xi$;
using two marked edges of a three-edge path gives $g$.

For the first binary implication, $\alpha$ inserts $x_1x_4$ on the
five-edge right side of $\iota_2$, and $\mu$ on
$x_0-x_1-x_4-x_5$ gives its query. The substitution
$(x_0,\ldots,x_5)=(a,b,a,b,c,d)$ gives
$\iota_2\Rightarrow\mu$. These also prove the stated consequences
of $\nu$ and their D-versions. Under $\eta$, an entire E part is marked;
$g$ then shares neighbours along consecutive two-edge subpaths. This
completes each marked component and gives $\alpha$ and $\iota_2$.
The conditional argument retains the same D-edge. Finally, $\omega$
makes $C(x,y)\approx D(x,y)$, which proves the last three binary
implications involving $\omega$. The laws $\rho$ and $\beta$ give
$x+y\preceq xy$ and $\xi$, respectively, whence
$xy\approx x+y+xy\approx x^2$.

We prove completeness by constructing a closed term for every proposed
nonconsequence. Let $E\subseteq B$ be closed under the displayed
implications. If $\tau\in E$, then $E=B$ and there is nothing to prove.
Otherwise use Table~\ref{tab:envelopes} to form an envelope of each
representative right side. A labelled linear term or edge is included
exactly when its label belongs to $E$; the original monomials are always
included. If the label in the final column belongs to $E$, use $0$
instead. The notation $L$ in the table means $L(a,b,c,d)$; in the last
row $P=\sum_{j=0}^4x_jx_{j+1}$. A list in a single entry consists of
separate insertions with the same label.

\begin{table}[!htbp]
\centering\footnotesize
\setlength{\tabcolsep}{4pt}
\caption{Envelopes of representative right sides. The entry $f:v$ means that $v$ is inserted when $f\in E$.}
\label{tab:envelopes}
\begin{tabular}{@{}>{\raggedright\arraybackslash}p{.27\textwidth}>{\raggedright\arraybackslash}p{.28\textwidth}>{\raggedright\arraybackslash}p{.29\textwidth}c@{}}
\toprule
Right side & New linear terms & New opposite-part edges & Zero label\\\midrule
$ab$ & $\rho:a,b$ & --- & $\chi$\\
$C(a,b)$ & $\omega:b$ & --- & $s$\\
$D(a,b)$ & --- & --- & $\xi$\\
$a+ab+bc$ & $\omega:b;\ \eta:c$ & --- & $s$\\
$D(a,b)+ac$ & $\lambda:c$ & --- & $\xi$\\
$L$ & $\rho:a,b,c,d$ & $\delta:ad$ & $\chi$\\
$a+L$ & $\omega:b,d;\ \eta:c$ & $\alpha:ad$ & $s$\\
$b+L$ & $\omega:a,c;\ \eta:d$ & $\mu:ad$ & $s$\\
$a+c+L$ & $\omega:b,d$ & $g:ad$ & $s$\\
$a+b+L$ & $\lambda:c,d$ & $\kappa:ad$ & $\xi$\\
$b+c+L$ & $\lambda:a,d$ & $\mu_2:ad$ & $\xi$\\
$a+b+c+L$ & $\lambda:d$ & $\varepsilon:ad$ & $\xi$\\
$C(a,b)+cd$ & $\omega:b;\ \zeta:c,d$ & --- & $s$\\
$D(a,b)+cd$ & $\zeta_D:c,d$ & --- & $\xi$\\
$C(a,b)+L(c,d,e,f)$ & $\omega:b;\ \zeta:c,d,e,f$ & $\nu:cf$ & $s$\\
$D(a,b)+L(c,d,e,f)$ & $\zeta_D:c,d,e,f$ & $\nu_D:cf$ & $\xi$\\
$C(a,b)+C(c,d)$ & $\omega:b,d$ & $\beta:ad,bc$ & $s$\\
$x_1+P$ & $\omega:x_0,x_2,x_4;\ \eta:x_3,x_5$ &
$\mu:x_0x_3;\ \alpha:x_1x_4;\ \iota_2:x_0x_5,x_2x_5$ & $s$\\
\bottomrule
\end{tabular}

\medskip
\begin{minipage}{.94\textwidth}\footnotesize
For a row with zero label $s$, one may add a disjoint $D(t,u)$.
Replace $\omega,\eta,\zeta,\alpha,\mu,g,\nu,\iota_2$ by their
D-versions and replace $s$ by $\xi$. The added D-edge has no new
nonzero monomial. Same-part edges, squares, and edges between distinct
components other than the two displayed $\beta$ edges have the zero
label. Purely linear right sides remain unchanged when $\tau\notin E$.
\end{minipage}
\end{table}
\FloatBarrier

Every listed insertion is a consequence of its label: the three-edge
rows are the defining laws, the two-edge marking rows use their defining
law in either orientation, and the independent-context rows keep that
context. In the five-edge row, $\mu$ gives $x_0x_3$, $\alpha$ gives
$x_1x_4$, and $\iota_2$ gives both indicated edges by
Lemma~\ref{lem:internalfold}. The zero entries follow by folding as in
Lemma~\ref{lem:elementaryfolds}. Thus each envelope is equivalent to its
original term under the laws in $E$.

Here are the closure conditions, including the possible interactions
between insertions. An unmarked component acquires marks only through
$\rho$, or through $\zeta,\zeta_D$ with the indicated context.
Once those marks occur, its chord is supplied by
$\rho\Rightarrow\delta$, $\zeta\Rightarrow\nu$, or
$\zeta_D\Rightarrow\nu_D$. On an S-component, the two possible
marking actions are even-part propagation by $\eta$ and opposite-part
propagation by $\omega$. The latter creates a D-context. A conditional
rule then has exactly its unconditional action, accounted for by
$\omega\Rightarrow\alpha,\eta$ and by
$\omega\nu_D\Rightarrow\nu$,
$\omega\zeta_D\Rightarrow\zeta$,
$\omega\xi\Rightarrow s$. Even-part propagation followed by a path
insertion is accounted for by $\eta g\Rightarrow\alpha,\iota_2$.
The same statements hold with a pre-existing D-context.

On a three-edge component there is only one missing opposite-part edge.
Its possible insertions are precisely its label and the stronger unary
premises in Table~\ref{tab:horn}; marking a D-component by $\lambda$
supplies all of them through $\kappa,\mu_2$. On the five-edge
component, the four possible missing opposite-part edges are exactly
those in the last row. Inserting $x_0x_3$ and $x_1x_4$ together supplies
the remaining edges exactly through
$\alpha\mu\Rightarrow\iota_2$. If new marks are present, the preceding
marking rules supply the same consequences. Thus no sequence of these
actions escapes the envelope. For two independent marked edges, the
only nonzero connection is the pair of $\beta$ edges; after it is added
the graph is a complete bipartite graph on the same parts. A new
opposite-part mark makes the term zero by
$\omega\beta\Rightarrow\omega\xi\Rightarrow s$.
All other cross-component connections, and all same-part edges, require
the zero label already specified in the table. This exhausts the edge,
marking, context and zero actions of every fixed law.

The path rules introduce no extra possibility. In every nonzero
envelope, two same-part marks in the same component are linked by a
sequence of marked vertices at distance two. For a D-component this
follows from the environment edges between opposite-part marks; for
the five-edge component it follows from the displayed even-part
marking pattern; all other components have at most four vertices.
Consequently any sharing produced by $\gamma_n$ is already produced
by $g$, and similarly in a D-context. This proves both saturation of
the envelopes and the last assertion of the lemma.

Theorem~\ref{thm:normal} now identifies these envelopes with the complete
closures. For every $f\notin E$ the query of $f$ is absent from its
nonzero envelope. The square and fresh-linear queries are covered by
the zero column and the purely linear case, respectively. The finite
algebra $B_X(E\cap\F_{25},\mathbf1_{g\in E},
\mathbf1_{g_D\in E})$ therefore satisfies every law in $E$ and fails
$f$. This excludes every additional implication, including any with
multiple premises.
\end{proof}

\subsection{The states of a canonical signature}

Separate the fixed representatives into the marking set
\[
M=\{\lambda,\eta_D,\eta,\omega_D,\omega,\zeta_D,\zeta,
\xi,\beta,s,\rho,\chi,\tau\}
\]
and the twelve edge representatives. The ordinary edge state is the
intersection with $\{\alpha,\mu,\iota_2,\nu,\delta\}$; its seven
possibilities are listed in Table~\ref{tab:edgestates}. When $q>0$,
the conditioned edge state contains $\varepsilon$ and has the ten
possibilities in the same table. The final column gives the minimum
conditioned state required by each ordinary state, with containment
understood as containment of sets of identities.

\begin{table}[ht]
\centering\small
\caption{Ordinary and conditioned edge states.}
\label{tab:edgestates}
\begin{tabular}{@{}cll@{\qquad}cl@{}}
\toprule
 & Ordinary state & Required & & Conditioned state\\\midrule
$O_0$ & $\varnothing$ & $D_0$ & $D_0$ & $\{\varepsilon\}$\\
$O_1$ & $\{\alpha\}$ & $D_3$ & $D_1$ & $\{\varepsilon,\kappa\}$\\
$O_2$ & $\{\mu\}$ & $D_6$ & $D_2$ & $\{\varepsilon,\mu_2\}$\\
$O_3$ & $\{\mu,\iota_2\}$ & $D_7$ & $D_3$ & $\{\varepsilon,\kappa,\alpha_D\}$\\
$O_4$ & $\{\alpha,\mu,\iota_2\}$ & $D_8$ & $D_4$ & $\{\varepsilon,\kappa,\mu_2\}$\\
$O_5$ & $O_4\cup\{\nu\}$ & $D_9$ & $D_5$ & $D_4\cup\{\alpha_D\}$\\
$O_6$ & $O_5\cup\{\delta\}$ & $D_9$ & $D_6$ & $D_4\cup\{\mu_D\}$\\
 & & & $D_7$ & $D_6\cup\{\iota_2^D\}$\\
 & & & $D_8$ & $D_7\cup\{\alpha_D\}$\\
 & & & $D_9$ & $D_8\cup\{\nu_D\}$\\
\bottomrule
\end{tabular}
\end{table}
\FloatBarrier

For the marking states put
\[
\begin{split}
L_0&=\varnothing,\qquad L_1=\{\lambda\},\qquad
L_2=L_1\cup\{\eta_D\},\\
L_3&=L_2\cup\{\omega_D\},\qquad
L_4=L_3\cup\{\zeta_D\},\qquad
L_5=L_4\cup\{\xi\}.
\end{split}
\]
The eighteen possibilities and their counts are given in
Table~\ref{tab:markstates}. The heading $0F$, for example, means
$p=0$ and $1\le q<\infty$; $F$ always denotes a positive finite
threshold and $I$ denotes $\infty$.

\begin{table}[!htbp]
\centering\small
\caption{Marking states and the number of allowed edge-state pairs. The last column also requires $\alpha$.}
\label{tab:markstates}
\begin{tabular}{@{}r l rrrrrrr@{}}
\toprule
 & $H\cap M$ & $00$ & $0F$ & $0I$ & $FF$ & $FI$ & $II$ & $II,\alpha$\\\midrule
0 & $L_0$ &5&5&10&6&14&25&8\\
1 & $L_1$ &1&2&6&3&10&20&7\\
2 & $L_2$ &1&0&2&0&4&12&6\\
3 & $L_2\cup\{\eta\}$ &1&0&2&0&0&4&4\\
4 & $L_3$ &0&0&2&0&4&12&6\\
5 & $L_3\cup\{\eta\}$ &0&0&2&0&0&4&4\\
6 & $L_4$ &0&0&1&0&2&7&4\\
7 & $L_3\cup\{\eta,\omega\}$ &0&0&0&0&0&3&3\\
8 & $L_4\cup\{\eta\}$ &0&0&1&0&0&3&3\\
9 & $L_5$ &0&0&1&0&2&7&4\\
10 & $L_5\cup\{\beta\}$ &0&0&0&0&0&7&4\\
11 & $L_5\cup\{\eta\}$ &0&0&1&0&0&3&3\\
12 & $L_5\cup\{\eta,\beta\}$ &0&0&0&0&0&3&3\\
13 & $L_4\cup\{\eta,\omega,\zeta\}$ &0&0&0&0&0&2&2\\
14 & $L_4\cup\{\eta,\omega,\zeta,\rho\}$ &0&0&0&0&0&1&1\\
15 & $L_5\cup\{\eta,\omega,\zeta,\beta,s\}$ &0&0&0&0&0&2&2\\
16 & $M\setminus\{\tau\}$ &0&0&0&0&0&1&1\\
17 & $M$ &0&0&0&0&0&1&1\\\midrule
 & Total &8&7&28&9&36&117&66\\
\bottomrule
\end{tabular}
\end{table}
\FloatBarrier

We give the selection rules and the arithmetic underlying this table.
Thus its entries can be obtained directly from the stated alternatives.
Besides the minimum D-state in Table~\ref{tab:edgestates}, impose
\begin{equation}\label{eq:stateconditions}
\begin{array}{rcl@{\qquad}rcl}
\lambda&\Rightarrow&D\supseteq D_4,&
\eta_D,\ q>0&\Rightarrow&D\supseteq D_8,\\
\omega_D&\Rightarrow&D\supseteq D_8,&
\zeta_D&\Rightarrow&D=D_9,\\
\eta,\ p>0&\Rightarrow&O\supseteq O_4,&
\omega&\Rightarrow&O\supseteq O_4,\\
\zeta&\Rightarrow&O\supseteq O_5,&
\rho&\Rightarrow&O=O_6,\\
\omega,\ D=D_9&\Rightarrow&O\supseteq O_5.&&&
\end{array}
\end{equation}
The set $H$ must also have the exact thresholds prescribed by
Theorem~\ref{thm:thresholds}. In particular, when $p=0$ only $O_0$
is allowed; when $0<p<\infty$ only $O_0,O_2$ are allowed. When
$0<q<\infty$, only $D_0,D_1,D_2,D_4,D_6$ are allowed. If $p=q=0$,
replace the D-list by
\[
\varnothing,\quad\{\varepsilon\},\quad
\{\varepsilon,\kappa\},\quad\{\varepsilon,\mu_2\},\quad
\{\varepsilon,\kappa,\mu_2\}.
\]
All marking states incompatible with the prescribed exact thresholds
are omitted. These requirements are necessary and sufficient: they are
exactly Table~\ref{tab:horn}, together with the threshold formulas,
written in the separated coordinates.

To see that there are exactly eighteen marking states, suppose first
that $\omega$ is absent. The conditional marking part is one of
$L_0,\ldots,L_5$. The mark $\eta$ can be added precisely to
$L_2,\ldots,L_5$, and $\beta$ can be added only to $L_5$, with or
without $\eta$. This gives twelve states. If $\omega$ is present but
$\xi$ is absent, the possibilities are
$L_3\cup\{\eta,\omega\}$,
$L_4\cup\{\eta,\omega,\zeta\}$, and the latter with $\rho$.
If both $\omega,\xi$ are present, $s,\beta,\zeta$ follow; one may then
adjoin neither $\rho$ nor $\chi$, both of them, or both together with
$\tau$. These are the remaining three states.

For $q>0$, the numbers of D-states containing
$D_0,D_3,D_6,D_7,D_8,D_9$ are respectively
\[
10,\quad4,\quad4,\quad3,\quad2,\quad1.
\]
Thus row 0 in region $II$ has
$10+4+4+3+2+1+1=25$ choices. Requiring $D\supseteq D_4$ changes
this sum to $6+3+4+3+2+1+1=20$. Requiring
$D\supseteq D_8$ gives $2+2+2+2+2+1+1=12$; requiring $D=D_9$
gives seven choices. If $\eta$ is also present, the ordinary states
reduce to $O_4,O_5,O_6$, giving $2+1+1=4$ or $1+1+1=3$.
If $\omega$ is present, $O_4$ cannot be paired with $D_9$, so the
first of those sums becomes $1+1+1=3$. The remaining ordinary
restrictions leave $O_5,O_6$ or only $O_6$, giving two or one choices.
This derives the entire $II$ column. Requiring $\alpha$ retains only
$O_1,O_4,O_5,O_6$, and gives its last column by the same sums.

In region $FI$, the ordinary choices are only $O_0,O_2$.
The four D requirements just used give respectively
$10+4=14$, $6+4=10$, $2+2=4$, and $1+1=2$.
The compatible marking states are $L_0,L_1,L_2,L_3,L_4,L_5$,
giving $14+10+4+4+2+2=36$.
In region $FF$, the five permitted D-states give five choices with
$O_0$ and one with $O_2$; adding $\lambda$ gives two and one.
There are no other marking states, so the total is $6+3=9$.
For $0I$, only $O_0$ remains, giving ten, six, two or one choices
under the respective D requirements; the ten permitted marking states
give $10+6+2+2+2+2+1+1+1+1=28$.
For $0F$, the counts are five and two. For $00$, the empty marking
state has five choices and each of $L_1,L_2,L_2\cup\{\eta\}$ has
one. This derives the remaining two totals, seven and eight.

\subsection{The six threshold regions}

\begin{theorem}[six-region classification]\label{thm:catalogue}
The legal fixed sets in each of the six disjoint threshold regions have the cardinalities in Table~\ref{tab:regions}. Thus the full lattice has 153 fixed nodes, 43 families indexed by a positive integer, and 9 families indexed by $1\le p\le q<\infty$.
\end{theorem}

\begin{table}[ht]
\centering\small
\caption{The six regions of normalized signatures.}
\label{tab:regions}
\begin{tabular}{@{}lrl@{}}
\toprule
Threshold region & Fixed sets $H$ & Contribution\\\midrule
$p=q=0$ & 8 & 8 fixed nodes\\
$p=0,\ 1\le q<\infty$ & 7 & 7 one-parameter families\\
$p=0,\ q=\infty$ & 28 & 28 fixed nodes\\
$1\le p\le q<\infty$ & 9 & 9 two-parameter families\\
$1\le p<\infty,\ q=\infty$ & 36 & 36 one-parameter families\\
$p=q=\infty$ & 117 & 117 fixed nodes\\
\bottomrule
\end{tabular}
\end{table}

\begin{proof}
Lemma~\ref{lem:finitecalculus} and the state selection rules give every
canonical fixed set, with no reference to a bound on the path indices.
The sums in Table~\ref{tab:markstates} are respectively
$8,7,28,9,36,117$. The derivation following that table establishes these
numbers by the seven ordinary states, the ten conditioned states, and
the eighteen marking states. The threshold theorem allows every actual
integer in each indicated finite region. Thus the counts are independent
of the particular positive finite values of $p,q$.

There are $8+28+117=153$ nodes with no free integer parameter,
$7+36=43$ one-parameter families, and nine families indexed by
$1\le p\le q<\infty$. The regions are disjoint and canonical signatures
are unique, so no node is counted twice.
\end{proof}

\subsection{Order and lattice operations}

Write $\sigma(U)$ for the canonical signature of a subvariety $U$.

\begin{proposition}\label{prop:lattice}
For normalized triples $U=(H,p,q)$ and $W=(K,r,s)$, identifying each with its variety,
\begin{align}
U\subseteq W&\iff H\supseteq K,\ p\ge r,\ q\ge s,\label{eq:order}\\
\sigma(U\wedge W)&=\Can(H\cup K,\max\{p,r\},\max\{q,s\}),\label{eq:meet}\\
\sigma(U\vee W)&=(H\cap K,\min\{p,r\},\min\{q,s\}).\label{eq:join}
\end{align}
Every lower cover of $(H,p,q)$ occurs among the following at most twenty-seven candidates: add one missing fixed representative; if $p<\infty$, replace $(p,q)$ by $(p+1,\max\{q,p+1\})$; if $q<\infty$, replace $q$ by $q+1$. Normalize, remove the original node and duplicates, and keep the inclusion-maximal remaining candidates.
\end{proposition}

\begin{proof}
Inclusion reverses the truth values of all representatives. Their completeness from Theorem~\ref{thm:reduction} gives \eqref{eq:order}. A meet adjoins the two sets of defining identities, yielding \eqref{eq:meet}. An identity holds in a join exactly when it holds in both factors. Its finite reduction is then true representative by representative in both factors, giving \eqref{eq:join}; this triple is automatically normalized.

Every strict lower node differs by an extra fixed representative or a strictly increased threshold. In the latter case the original threshold is finite, so increasing it by one does not exceed the target. Thus every strict lower node lies below one of the stated strict candidates. A lower cover must itself be such a candidate. If a maximal candidate had an intermediate node below $U$, that intermediate node would lie below a larger strict candidate, a contradiction. This proves the cover formula for the entire lattice, including infinite thresholds.
\end{proof}

\begin{corollary}\label{cor:orderproperties}
The lattice $\Sub(\mathcal V)$ satisfies the ascending chain condition and has no infinite antichain. For $m\ge1$, the number of nodes whose exact thresholds belong to $\{0,1,\ldots,m,\infty\}$ is
\[
153+43m+\frac{9m(m+1)}2.
\]
\end{corollary}

\begin{proof}
Along an ascending chain, the finite truth set can strictly decrease only finitely often. Each threshold is nonincreasing in $\mathbb N_0\cup\{\infty\}$ and cannot decrease infinitely often. For an infinite antichain, restrict to an infinite subfamily with the same finite truth set and with the same pattern of infinite coordinates. With at most one finite coordinate the members form a chain. With two finite coordinates, an infinite collection of distinct pairs has comparable members: if one first coordinate occurs infinitely often, order the second coordinates; otherwise choose an increasing sequence of first coordinates, along which incomparability would force an impossible infinite decreasing sequence of second coordinates. This proves the antichain assertion. The counting formula follows directly from Table~\ref{tab:regions}.
\end{proof}
\section{The three finite subvariety lattices}
\label{sec:finitelattices}

The finite lattices follow by imposing their ambient identities on the
states already classified. In particular, their sizes are consequences
of the same implication calculus.

\begin{theorem}\label{thm:smalllattices}
The lattices $\Sub(\V(R_{01}))$, $\Sub(\V(R_{02}))$ and
$\Sub(\V(R_{11}))$ have respectively $11$, $11$ and $66$ elements.
Their complete relative bases and lower covers are listed in
Tables~\ref{tab:finite01}, \ref{tab:finite02} and \ref{tab:finite11}.
Exactly $5$, $5$ and $18$ of these nodes, respectively, satisfy $\delta$.
\end{theorem}

\begin{proof}
The bases in Theorem~\ref{thm:bases} give
\begin{align}
\V(R_{01})&=\mathcal V[\xi,\alpha,\beta],\label{eq:r01ambient}\\
\V(R_{02})&=\mathcal V[\xi,\eta],\label{eq:r02ambient}\\
\V(R_{11})&=\mathcal V[\alpha].\label{eq:r11ambient}
\end{align}
For $R_{11}$, the law $\alpha$ makes both thresholds infinite. The
ordinary edge state must be one of $O_1,O_4,O_5,O_6$.
The last column of Table~\ref{tab:markstates} therefore gives
\[
8+7+6+4+6+4+4+3+3+4+4+3+3+2+1+2+1+1=66.
\]
Every marking state admits exactly one state pair containing $\delta$:
it is $(O_6,D_9)$. Hence eighteen of these nodes satisfy $\delta$.

For $R_{01}$, the law $\beta$ already implies $\xi$ and all path laws,
and $\alpha$ must also be present. The possible marking rows are
$10,12,15,16,17$ in Table~\ref{tab:markstates}. Their contributions in
the last column are $4,3,2,1,1$, giving eleven nodes. Each row again has
exactly one pair with $\delta$, giving five.

For $R_{02}$, the law $\xi$ makes $q=\infty$. Since $\eta$ is
present, Theorem~\ref{thm:thresholds} gives either $p=0$ or $p=\infty$.
When $p=0$, the only permitted marking row containing $\xi,\eta$ is
row 11, and its contribution in region $0I$ is one. When $p=\infty$,
the possible rows are $11,12,15,16,17$, with respective contributions
$3,3,2,1,1$. This gives $1+3+3+2+1+1=11$ nodes. The node with $p=0$
does not satisfy $\delta$; in each of the other five rows precisely
$(O_6,D_9)$ does. Thus the number satisfying $\delta$ is five.

For each state the indicated identities are a relative basis.
Table~\ref{tab:horn} removes redundant identities to give the shorter
bases in the appendix. Conversely, closing any displayed basis by that
table and applying the threshold formulas recovers its state. Distinct
states differ on a representative; the envelope and algebra in
Lemma~\ref{lem:finitecalculus} separate them. Thus every counted state
defines a distinct variety and there are no others. Inclusion and
covers follow from Proposition~\ref{prop:lattice}.
\end{proof}

\begin{remark}
A finite relative basis over $\mathcal V$ need not be an absolute finite
identity basis, since $\mathcal V$ is nonfinitely based. The next section
gives the additional argument needed for that distinction.
\end{remark}
\section{Finite bases and the unique limit subvariety}
\label{sec:finitebasis}

Put $\mathcal L=\V(SR_6)=\V(R_{00})$ and $\mathcal A=\V(R_{11})$. We prove a criterion for every subvariety of $\mathcal V$, so that the finite-basis classification does not rely on extrapolating a finite list of examples.

\begin{lemma}\label{lem:srcriterion}
In $SR_6$, a term is identically $0$ exactly when it is basic-zero or has a marked vertex. Every remaining term absorbs precisely its displayed monomials. In particular, an unmarked bipartite term with disjoint isolated linear summands is rigid, and
\[
\mathcal L=\mathcal V[s].
\]
\end{lemma}

\begin{proof}
No vertex of the nonzero-product path in $R_{00}$ lies below $c$. Hence a linear summand sharing its variable with a quadratic summand forces value $0$. A basic-zero term also has value $0$. Conversely, an unmarked bipartite graph maps to the middle edge $a-b$, and its isolated linear variables can be sent to $c$, giving nonzero value $c$. If the graph is empty, assigning all variables to $c$ also gives $c$.

For a new linear query in an unmarked graph, move its vertex to an end of $P_4$; for a same-part or cross-component edge, orient the components to make the query product zero. A missing opposite-part edge is separated by mapping its endpoints to $p,q$ and the remaining vertices of their parts to $a,b$. A graph-external variable may be sent to $0$, or to $c$ if it is a displayed isolated linear summand used in a quadratic query. These assignments, and the purely linear assignment from Theorem~\ref{thm:criterion}, show that only displayed monomials are absorbed. Two-way absorption proves rigidity.

The law $s$ makes every marked edge $x+xy$ zero, and so makes exactly the additional marked terms zero in $\mathcal V$. Theorem~\ref{thm:normal} with $H=\{s\}$ leaves every other unmarked bipartite term unchanged. It therefore has exactly the criterion just proved for $SR_6$, establishing $\mathcal L=\mathcal V[s]$.
\end{proof}

\begin{theorem}[uniform finite-basis criterion]\label{thm:fbcriterion}
For every $\mathcal U\subseteq\mathcal V$, the following conditions are equivalent:
\begin{enumerate}[label=(\roman*),leftmargin=2.2em]
\item $\mathcal U$ is finitely based;
\item $\mathcal U\models\delta$;
\item $\mathcal U\subseteq\mathcal V[\delta]=\mathcal A[\delta]$;
\item $\mathcal L\nsubseteq\mathcal U$.
\end{enumerate}
If $\mathcal U=\mathcal V[H,p,q]$ satisfies these conditions, a finite basis consists of the commutative ai-semiring axioms together with
\begin{equation}\label{eq:finitebasis}
x\preceq y^2,\qquad x^2\approx yzt,\qquad \delta,\qquad H.
\end{equation}
Here $\mathcal V[H,p,q]$ abbreviates the presentation \eqref{eq:presentation}.
\end{theorem}

\begin{proof}
First suppose $\mathcal L\nsubseteq\mathcal U$. There is an identity true in $\mathcal U$ and false in $SR_6$; select a monomial absorption $q\preceq v$ with the same property. Lemma~\ref{lem:srcriterion} implies that $v$ is an unmarked bipartite term with disjoint isolated linear summands, and $q$ is absent. If its graph is empty, collapsing variables gives $\tau$ as in Lemma~\ref{lem:elementaryfolds}.

If the graph has an edge and $q$ is a graph vertex as a linear query, fold the two parts to the endpoints of an edge, retaining that queried endpoint, and map isolated linear variables to the product. This extracts $\rho$. For a missing opposite-part edge within one component, retain the endpoints as $a,d$ and fold the remaining two parts to $c,b$. The right side maps into $ab+bc+cd$, giving $\delta$. In every other case, orient components freely to make a quadratic query a square, or map a graph-external query variable to a quadratic term so that the query becomes zero. The right side folds to a single edge; thus the extracted law is $\chi$. This also covers an already quadratic square or a higher monomial query.

Each of $\tau,\chi$ implies $\delta$. In $\mathcal V$, so does $\rho$: the path term $ab+bc+cd$ then absorbs both $a$ and $d$, and the environment law $\theta_1$ gives the edge $ad$. Therefore (iv) implies (ii).

The equivalence of (ii) and (iii) follows since $\delta$ implies $\alpha$ by addition of a marked endpoint, so $\mathcal V[\delta]=\mathcal A[\delta]$. To prove (ii) implies (i), use Corollary~\ref{cor:presentation} to choose a finite fixed part $H$ and two thresholds. Repeated three-edge shortening by $\delta$ yields the endpoint absorption for every odd walk. Identifying the first and last vertices of an odd closed walk makes its square absorbed, hence the whole odd-cycle term zero. Thus $\delta$ yields all $\sigma_n$ and all $\theta_n$. It also yields every $\gamma_n$ and $\gamma_n^D$, because their query endpoints are linked by an odd path. Consequently \eqref{eq:finitebasis} implies the full environment basis and both path families. Conversely, all its laws hold in $\mathcal U$. This proves it is a finite absolute basis.

For the remaining implication, assume $\mathcal L\subseteq\mathcal U\subseteq\mathcal V$ and suppose $F$ is a finite basis of $\mathcal U$. Choose an odd cycle length larger than three and than the number of variables in every member of $F$. Its polynomial $w$ satisfies $w\approx w+z$ in $\mathcal U$, because $\mathcal U\subseteq\mathcal V$. At a first polynomial-changing rewrite, the source side $s$ of the basis identity must be purely linear if a nonempty multiplicative context is present. With no such context, a square or higher monomial cannot map into $w$, a variable occurring both linearly and quadratically is excluded by degree, and a short odd cycle in $s$ would map to a shorter odd closed walk in the chosen long cycle. Thus $s$ is unmarked and bipartite with disjoint isolated linear variables, exactly as in Theorem~\ref{thm:nfb}.

Every identity of $F$ holds in $\mathcal L$. Lemma~\ref{lem:srcriterion} therefore makes that rewrite a polynomial equality, contradicting its choice. This proves that every variety in $[\mathcal L,\mathcal V]$ is nonfinitely based, so (i) implies (iv). Finally, $SR_6$ fails $\delta$: assign the three-edge path to $p-b-a-q$, giving right-side value $c$ and query value $pq=0$. This also directly confirms that (ii) excludes containment of $\mathcal L$.
\end{proof}

\begin{corollary}\label{cor:size}
Exactly eighteen subvarieties of $\mathcal V$ are finitely based. All other subvarieties are nonfinitely based, and $\mathcal L$ is its unique limit subvariety. The varieties generated by $R_{01},R_{02},R_{11}$ have respectively $5,5,18$ finitely based subvarieties, and each has the same unique limit subvariety $\mathcal L$.
\end{corollary}

\begin{proof}
The finite-basis nodes of $\mathcal V$ are precisely the subvarieties of $\mathcal A[\delta]$. Theorem~\ref{thm:smalllattices} and Table~\ref{tab:finite11} give exactly eighteen such nodes. The other two counts follow from the same $\delta$ columns in their complete tables.

Theorem~\ref{thm:fbcriterion} makes $\mathcal L$ itself nonfinitely based and every proper subvariety of it finitely based. Every nonfinitely based subvariety of $\mathcal V$ contains $\mathcal L$; if it is larger, it has a proper nonfinitely based subvariety and cannot be a limit variety. This proves uniqueness. All three smaller ambient varieties contain $\mathcal L$, as follows either from their bases or from Lemma~\ref{lem:srcriterion}. The containments are strict: setting $x=b,y=p$ makes $x+xy=c\ne0=x^2$ in each of the four new semirings, so they all fail $s$.
\end{proof}
\section{The sixteen-node interval above \texorpdfstring{$\mathcal V[\mu]$}{V[mu]}}
\label{sec:upperinterval}

Put $\mathcal P=\mathcal V[\mu]$. The interval above this variety illustrates how the global reduction, finite closure and explicit separating algebra work together.

\begin{lemma}\label{lem:seven}
Every identity true in $\mathcal P$ is equivalent, relative to $\mathcal V$, to a subset of
\[
B=\{\varepsilon,\gamma_1^D,\gamma_1,\kappa,\mu_2,\mu_D,\mu\}.
\]
\end{lemma}

\begin{proof}
Theorem~\ref{thm:reduction} reduces every identity to fixed and path representatives. Every representative in the reduction must hold in $\mathcal P$. On the twenty-five fixed right sides, the closure $C_{\{\mu\},0,0}$ gives exactly the truth set
\[
\{\varepsilon,\kappa,\mu_2,\mu_D,\mu\}.
\]
Indeed, Table~\ref{tab:horn} closes $\{\mu\}$ to precisely these
five fixed identities and $g,g_D$. Its completeness excludes every
other fixed representative. Theorem~\ref{thm:thresholds} gives the exact path thresholds $(1,1)$, ruling out $\gamma_2,\gamma_2^D$ and every larger index. Only the seven listed representatives remain.
\end{proof}

Their positive implications are generated by
\begin{equation}\label{eq:upperhorn}
\begin{aligned}
\gamma_1^D&\Rightarrow\varepsilon,&
\gamma_1&\Rightarrow\gamma_1^D,\\
\kappa&\Rightarrow\varepsilon,&
\mu_2&\Rightarrow\varepsilon,\\
\mu_D&\Rightarrow\gamma_1^D,\kappa,\mu_2,&
\mu&\Rightarrow\gamma_1,\mu_D.
\end{aligned}
\end{equation}
For $\gamma_1^D\Rightarrow\varepsilon$, substitute the double-marked edge $D(a,b)$ as context on the path $a-b-c-d$. For the consequences of $\mu_D$, use $D(b,c)$ to obtain $\mu_2$, use $D(a,b)$ to obtain $\kappa$, and reverse the path to obtain $\gamma_1^D$. The remaining implications add context or a marked linear summand. Thus every rule in \eqref{eq:upperhorn} is an equational consequence in the stated environment.

\begin{theorem}\label{thm:interval}
The interval $[\mathcal P,\mathcal V]$ consists of the sixteen nodes in Table~\ref{tab:uppercatalogue}, with twenty-five cover edges. The rules \eqref{eq:upperhorn} are complete, including implications with multiple premises. The interval is distributive, and
\[
\mathcal V[\gamma_1,\mu_D]\succ\mathcal P,\qquad
\mathcal V[\gamma_1,\mu_D]=\mathcal P\vee\V(R_{01}).
\]
The displayed node is the unique upper cover of $\mathcal P$. Exactly seven nodes in the interval contain $\mathcal A=\V(R_{11})$; the remaining nine are incomparable with $\mathcal A$.
\end{theorem}

\begin{proof}
Every subset closed under \eqref{eq:upperhorn} is also closed under
Table~\ref{tab:horn} when viewed as a subset of its seven symbols:
no other rule in that table has its premises inside this set and its
conclusion outside it. Lemma~\ref{lem:finitecalculus} therefore proves
that no further implication, with any number of premises, holds.

Use the coordinate order
$(\varepsilon,\gamma_1^D,\gamma_1,\kappa,\mu_2,\mu_D,\mu)$
in Table~\ref{tab:uppercatalogue}. The empty
set is one choice. Every nonempty set contains $\varepsilon$. If it
contains neither $\mu_D$ nor $\mu$, the three symbols
$\gamma_1^D,\kappa,\mu_2$ are independent, while $\gamma_1$ may be
chosen only when $\gamma_1^D$ has been chosen. There are consequently
$3\cdot2\cdot2=12$ such nonempty sets. If $\mu_D$ is present but
$\mu$ is absent, then $\gamma_1^D,\kappa,\mu_2$ are forced and
$\gamma_1$ is optional, giving two more. Finally, $\mu$ forces all
seven symbols, giving one. The total is $1+12+2+1=16$.

Equivalently, these closed sets are the order ideals of the poset whose
cover relations, from smaller to larger elements, are
\[
\varepsilon<\gamma_1^D,\kappa,\mu_2,\qquad
\gamma_1^D<\gamma_1,\mu_D,\qquad
\kappa,\mu_2<\mu_D,\qquad
\gamma_1,\mu_D<\mu.
\]
The twelve nonempty sets omitting $\mu_D,\mu$ form
$C_3\times C_2\times C_2$, where $C_m$ is the $m$-element chain. They have
$2\cdot2\cdot2+3\cdot1\cdot2+3\cdot2\cdot1=20$ cover edges.
One more edge joins the empty set to $\{\varepsilon\}$. The two sets
containing $\mu_D$ and omitting $\mu$ each have one additional lower
neighbour obtained by removing $\mu_D$; they also have an edge between
them. One final edge joins the full set to the set omitting $\mu$.
Thus the total is $20+1+2+1+1=25$. Reversing inclusion gives the
variety covers in Table~\ref{tab:uppercatalogue}.

Order ideals are closed under union and intersection, so this interval
is distributive. The full ideal has just one maximal proper ideal,
namely the first six coordinates, proving the assertion about the
unique upper cover of $\mathcal P$. Lemma~\ref{lem:seven} excludes
any further variety in the interval. Distinct ideals are separated by
the corresponding nonzero envelope in Lemma~\ref{lem:finitecalculus}
and its algebra $B_X$.

The common representative truth set of $\mathcal P$ and $R_{01}$ consists of the first six coordinates of $B$; their exact common thresholds are $(1,1)$. Indeed, $R_{01}$ satisfies all of $B$ except $\mu$, as is also given by its complete normal-form signature. The join formula \eqref{eq:join} therefore gives $\mathcal P\vee\V(R_{01})=\mathcal V[\gamma_1,\mu_D]$. The truth set of $\mathcal A$ within $B$ is $\{\varepsilon,\gamma_1^D,\gamma_1,\kappa\}$, so the varieties containing it correspond to the ideals of
the four-element poset on
$\{\varepsilon,\gamma_1^D,\gamma_1,\kappa\}$.
Besides the empty ideal there are $3\cdot2=6$ choices, giving seven. Every other row contains $\mathcal P$ and fails to contain $\mathcal A$. Since $\mathcal P$ fails $\alpha$, no such row can be contained in $\mathcal A$ either, proving incomparability.
\end{proof}

\begin{table}[ht]
\centering\small
\caption{All nodes of $[\mathcal P,\mathcal V]$.}
\label{tab:uppercatalogue}
\begin{tabular}{@{}l l c l@{}}
\toprule Node & Relative basis & Truth vector & Lower covers\\\midrule
P00 & $\varnothing$ & \texttt{0000000} & P01\\
P01 & $\{\varepsilon\}$ & \texttt{1000000} & P02, P03, P04\\
P02 & $\{\gamma_1^D\}$ & \texttt{1100000} & P05, P06, P07\\
P03 & $\{\kappa\}$ & \texttt{1001000} & P06, P08\\
P04 & $\{\mu_2\}$ & \texttt{1000100} & P07, P08\\
P05 & $\{\gamma_1\}$ & \texttt{1110000} & P09, P10\\
P06 & $\{\gamma_1^D,\kappa\}$ & \texttt{1101000} & P09, P11\\
P07 & $\{\gamma_1^D,\mu_2\}$ & \texttt{1100100} & P10, P11\\
P08 & $\{\kappa,\mu_2\}$ & \texttt{1001100} & P11\\
P09 & $\{\gamma_1,\kappa\}$ & \texttt{1111000} & P12\\
P10 & $\{\gamma_1,\mu_2\}$ & \texttt{1110100} & P12\\
P11 & $\{\gamma_1^D,\kappa,\mu_2\}$ & \texttt{1101100} & P12, P13\\
P12 & $\{\gamma_1,\kappa,\mu_2\}$ & \texttt{1111100} & P14\\
P13 & $\{\mu_D\}$ & \texttt{1101110} & P14\\
P14 & $\{\gamma_1,\mu_D\}$ & \texttt{1111110} & P15\\
P15 & $\{\mu\}$ & \texttt{1111111} & ---\\
\bottomrule
\end{tabular}
\end{table}

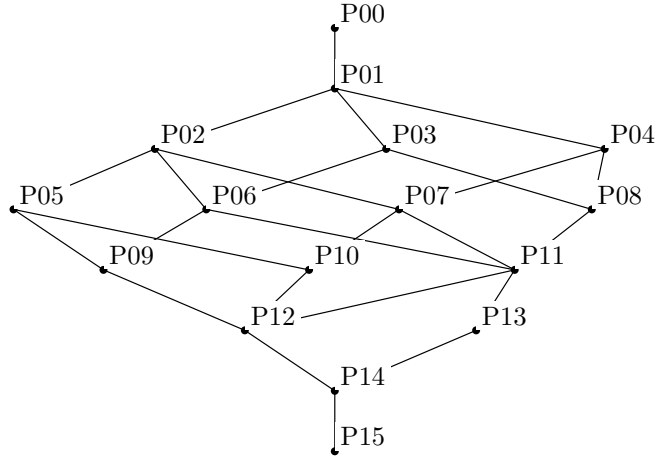
\begin{figure}[ht]
\centering
\begin{tikzpicture}[x=1.7cm,y=0.8cm]
\coordinate (P00) at (0,7);
\coordinate (P01) at (0,6);
\coordinate (P02) at (-1.4,5);
\coordinate (P03) at (0.4,5);
\coordinate (P04) at (2.1,5);
\coordinate (P05) at (-2.5,4);
\coordinate (P06) at (-1,4);
\coordinate (P07) at (0.5,4);
\coordinate (P08) at (2,4);
\coordinate (P09) at (-1.8,3);
\coordinate (P10) at (-0.2,3);
\coordinate (P11) at (1.4,3);
\coordinate (P12) at (-0.7,2);
\coordinate (P13) at (1.1,2);
\coordinate (P14) at (0,1);
\coordinate (P15) at (0,0);
\draw[line width=0.45pt] (P00) -- (P01);
\draw[line width=0.45pt] (P01) -- (P02);
\draw[line width=0.45pt] (P01) -- (P03);
\draw[line width=0.45pt] (P01) -- (P04);
\draw[line width=0.45pt] (P02) -- (P05);
\draw[line width=0.45pt] (P02) -- (P06);
\draw[line width=0.45pt] (P02) -- (P07);
\draw[line width=0.45pt] (P03) -- (P06);
\draw[line width=0.45pt] (P03) -- (P08);
\draw[line width=0.45pt] (P04) -- (P07);
\draw[line width=0.45pt] (P04) -- (P08);
\draw[line width=0.45pt] (P05) -- (P09);
\draw[line width=0.45pt] (P05) -- (P10);
\draw[line width=0.45pt] (P06) -- (P09);
\draw[line width=0.45pt] (P06) -- (P11);
\draw[line width=0.45pt] (P07) -- (P10);
\draw[line width=0.45pt] (P07) -- (P11);
\draw[line width=0.45pt] (P08) -- (P11);
\draw[line width=0.45pt] (P09) -- (P12);
\draw[line width=0.45pt] (P10) -- (P12);
\draw[line width=0.45pt] (P11) -- (P12);
\draw[line width=0.45pt] (P11) -- (P13);
\draw[line width=0.45pt] (P12) -- (P14);
\draw[line width=0.45pt] (P13) -- (P14);
\draw[line width=0.45pt] (P14) -- (P15);
\fill (P00) circle (1.5pt);
\node[anchor=south west,inner sep=2pt,font=\small,fill=white] at (P00) {P00};
\fill (P01) circle (1.5pt);
\node[anchor=south west,inner sep=2pt,font=\small,fill=white] at (P01) {P01};
\fill (P02) circle (1.5pt);
\node[anchor=south west,inner sep=2pt,font=\small,fill=white] at (P02) {P02};
\fill (P03) circle (1.5pt);
\node[anchor=south west,inner sep=2pt,font=\small,fill=white] at (P03) {P03};
\fill (P04) circle (1.5pt);
\node[anchor=south west,inner sep=2pt,font=\small,fill=white] at (P04) {P04};
\fill (P05) circle (1.5pt);
\node[anchor=south west,inner sep=2pt,font=\small,fill=white] at (P05) {P05};
\fill (P06) circle (1.5pt);
\node[anchor=south west,inner sep=2pt,font=\small,fill=white] at (P06) {P06};
\fill (P07) circle (1.5pt);
\node[anchor=south west,inner sep=2pt,font=\small,fill=white] at (P07) {P07};
\fill (P08) circle (1.5pt);
\node[anchor=south west,inner sep=2pt,font=\small,fill=white] at (P08) {P08};
\fill (P09) circle (1.5pt);
\node[anchor=south west,inner sep=2pt,font=\small,fill=white] at (P09) {P09};
\fill (P10) circle (1.5pt);
\node[anchor=south west,inner sep=2pt,font=\small,fill=white] at (P10) {P10};
\fill (P11) circle (1.5pt);
\node[anchor=south west,inner sep=2pt,font=\small,fill=white] at (P11) {P11};
\fill (P12) circle (1.5pt);
\node[anchor=south west,inner sep=2pt,font=\small,fill=white] at (P12) {P12};
\fill (P13) circle (1.5pt);
\node[anchor=south west,inner sep=2pt,font=\small,fill=white] at (P13) {P13};
\fill (P14) circle (1.5pt);
\node[anchor=south west,inner sep=2pt,font=\small,fill=white] at (P14) {P14};
\fill (P15) circle (1.5pt);
\node[anchor=south west,inner sep=2pt,font=\small,fill=white] at (P15) {P15};
\end{tikzpicture}
\caption{The complete interval $[\mathcal P,\mathcal V]$; node labels agree with Table~\ref{tab:uppercatalogue}.}
\label{fig:interval}
\end{figure}
\Needspace{15\baselineskip}
\section{Conclusion and open problems}

We have classified the compatible additions on the common multiplication
of $SR_6$ and $TR_6$ and studied the four remaining isomorphism types.
Each has an explicit infinite identity basis and is nonfinitely based.
Their subvariety lattices have $11$, $11$, $66$ and countably infinitely
many elements. The last lattice is described by a finite implication
calculus and two exact path thresholds. The classification of its
ordinary, conditioned and marking states gives the six-region count,
while the uniform finite-basis criterion identifies eighteen finitely
based nodes and the unique limit subvariety $\V(SR_6)$.

The strong nonfinite-basis status of the four semirings remains open.
It would also be useful to determine sharp bounds for the sizes of
finite algebras separating distinct signatures, and to compare the
present graph reductions with those for other varieties of
semilattice-ordered semigroups.
\FloatBarrier

\Needspace{7\baselineskip}
\bigskip
\paragraph{Acknowledgments.}
This work was supported by the National Natural Science Foundation of China (12371024), the Science and Technology Research Program of Chongqing Municipal Education Commission (KJZD-K2024011102) and the Chongqing Natural Science Foundation Innovation and Development Joint Fund (Municipal Education Commission) (CSTB2025NSCQ-LZX0067).

\clearpage
\appendix
\section{Complete finite catalogues}
\label{app:catalogues}
The following tables list every node of the three finite lattices.
A row with relative basis $E$ denotes $\V(R_{ij})[E]$.
The column $\delta$ is its exact truth value; Theorem~\ref{thm:fbcriterion}
identifies this column with finite basability. Covers are lower covers
and refer to row numbers in the same table. Row numbers are local labels.
Each full signature is recovered from its displayed basis using
Table~\ref{tab:horn} and Theorem~\ref{thm:thresholds}.
\Needspace{18\baselineskip}
{\small\renewcommand{\arraystretch}{1.1}
\begin{longtable}{@{}r p{0.43\textwidth} c p{0.27\textwidth}@{}}
\caption{All subvarieties of $\V(R_{01})$.}\label{tab:finite01}\\
\toprule Row & Relative basis $E$ & $\delta$ & Lower covers\\\midrule\endfirsthead
\multicolumn{4}{c}{Table \thetable\ (continued)}\\
\toprule Row & Relative basis $E$ & $\delta$ & Lower covers\\\midrule\endhead
\bottomrule\endfoot
1 & $\varnothing$ & 0 & 2\\
2 & $\{\iota_2\}$ & 0 & 3, 4\\
3 & $\{\eta\}$ & 0 & 6\\
4 & $\{\nu\}$ & 0 & 5, 6\\
5 & $\{\delta\}$ & 1 & 7\\
6 & $\{\eta,\nu\}$ & 0 & 7, 8\\
7 & $\{\eta,\delta\}$ & 1 & 9\\
8 & $\{s\}$ & 0 & 9\\
9 & $\{\delta,s\}$ & 1 & 10\\
10 & $\{\chi\}$ & 1 & 11\\
11 & $\{\tau\}$ & 1 & ---\\
\end{longtable}
}
\Needspace{18\baselineskip}
{\small\renewcommand{\arraystretch}{1.1}
\begin{longtable}{@{}r p{0.43\textwidth} c p{0.27\textwidth}@{}}
\caption{All subvarieties of $\V(R_{02})$.}\label{tab:finite02}\\
\toprule Row & Relative basis $E$ & $\delta$ & Lower covers\\\midrule\endfirsthead
\multicolumn{4}{c}{Table \thetable\ (continued)}\\
\toprule Row & Relative basis $E$ & $\delta$ & Lower covers\\\midrule\endhead
\bottomrule\endfoot
1 & $\varnothing$ & 0 & 2\\
2 & $\{\gamma_1\}$ & 0 & 3, 4\\
3 & $\{\beta\}$ & 0 & 5\\
4 & $\{\nu\}$ & 0 & 5, 6\\
5 & $\{\nu,\beta\}$ & 0 & 7, 8\\
6 & $\{\delta\}$ & 1 & 7\\
7 & $\{\delta,\beta\}$ & 1 & 9\\
8 & $\{s\}$ & 0 & 9\\
9 & $\{\delta,s\}$ & 1 & 10\\
10 & $\{\chi\}$ & 1 & 11\\
11 & $\{\tau\}$ & 1 & ---\\
\end{longtable}
}
\clearpage
{\small\renewcommand{\arraystretch}{1.05}
\begin{longtable}{@{}r p{0.43\textwidth} c p{0.27\textwidth}@{}}
\caption{All subvarieties of $\V(R_{11})$.}\label{tab:finite11}\\
\toprule Row & Relative basis $E$ & $\delta$ & Lower covers\\\midrule\endfirsthead
\multicolumn{4}{c}{Table \thetable\ (continued)}\\
\toprule Row & Relative basis $E$ & $\delta$ & Lower covers\\\midrule\endhead
\bottomrule\endfoot
1 & $\varnothing$ & 0 & 2\\
2 & $\{\mu_2\}$ & 0 & 3, 4\\
3 & $\{\lambda\}$ & 0 & 5\\
4 & $\{\iota_2^D\}$ & 0 & 5, 6, 8\\
5 & $\{\lambda,\iota_2^D\}$ & 0 & 7, 9, 12\\
6 & $\{\nu_D\}$ & 0 & 9, 13\\
7 & $\{\eta_D\}$ & 0 & 10, 11, 14\\
8 & $\{\iota_2\}$ & 0 & 12, 13\\
9 & $\{\lambda,\nu_D\}$ & 0 & 10, 16\\
10 & $\{\eta_D,\nu_D\}$ & 0 & 15, 20\\
11 & $\{\omega_D\}$ & 0 & 15, 21\\
12 & $\{\lambda,\iota_2\}$ & 0 & 14, 16\\
13 & $\{\nu_D,\iota_2\}$ & 0 & 16, 17\\
14 & $\{\eta_D,\iota_2\}$ & 0 & 19, 20, 21\\
15 & $\{\omega_D,\nu_D\}$ & 0 & 22, 28\\
16 & $\{\lambda,\nu_D,\iota_2\}$ & 0 & 20, 23\\
17 & $\{\nu\}$ & 0 & 18, 23\\
18 & $\{\delta\}$ & 1 & 24\\
19 & $\{\eta\}$ & 0 & 25, 26\\
20 & $\{\eta_D,\nu_D,\iota_2\}$ & 0 & 25, 27, 28\\
21 & $\{\omega_D,\iota_2\}$ & 0 & 26, 28\\
22 & $\{\zeta_D\}$ & 0 & 29, 36\\
23 & $\{\lambda,\nu\}$ & 0 & 24, 27\\
24 & $\{\lambda,\delta\}$ & 1 & 31\\
25 & $\{\eta,\nu_D\}$ & 0 & 32, 33\\
26 & $\{\eta,\omega_D\}$ & 0 & 33, 34\\
27 & $\{\eta_D,\nu\}$ & 0 & 31, 32, 35\\
28 & $\{\omega_D,\nu_D,\iota_2\}$ & 0 & 33, 35, 36\\
29 & $\{\xi\}$ & 0 & 30, 42\\
30 & $\{\beta\}$ & 0 & 43\\
31 & $\{\eta_D,\delta\}$ & 1 & 37, 38\\
32 & $\{\eta,\nu\}$ & 0 & 37, 39\\
33 & $\{\eta,\omega_D,\nu_D\}$ & 0 & 39, 40\\
34 & $\{\omega\}$ & 0 & 46\\
35 & $\{\omega_D,\nu\}$ & 0 & 38, 39, 41\\
36 & $\{\zeta_D,\iota_2\}$ & 0 & 40, 41, 42\\
37 & $\{\eta,\delta\}$ & 1 & 44\\
38 & $\{\omega_D,\delta\}$ & 1 & 44, 45\\
39 & $\{\eta,\omega_D,\nu\}$ & 0 & 44, 46, 47\\
40 & $\{\eta,\zeta_D\}$ & 0 & 47, 48\\
41 & $\{\zeta_D,\nu\}$ & 0 & 45, 47, 49\\
42 & $\{\xi,\iota_2\}$ & 0 & 43, 48, 49\\
43 & $\{\beta,\iota_2\}$ & 0 & 50, 51\\
44 & $\{\eta,\omega_D,\delta\}$ & 1 & 52, 53\\
45 & $\{\zeta_D,\delta\}$ & 1 & 53, 54\\
46 & $\{\omega,\nu_D\}$ & 0 & 52, 59\\
47 & $\{\eta,\zeta_D,\nu\}$ & 0 & 53, 55, 59\\
48 & $\{\eta,\xi\}$ & 0 & 50, 55\\
49 & $\{\nu,\xi\}$ & 0 & 51, 54, 55\\
50 & $\{\eta,\beta\}$ & 0 & 57\\
51 & $\{\nu,\beta\}$ & 0 & 56, 57\\
52 & $\{\omega,\delta\}$ & 1 & 61\\
53 & $\{\eta,\zeta_D,\delta\}$ & 1 & 58, 61\\
54 & $\{\delta,\xi\}$ & 1 & 56, 58\\
55 & $\{\eta,\nu,\xi\}$ & 0 & 57, 58\\
56 & $\{\delta,\beta\}$ & 1 & 60\\
57 & $\{\eta,\nu,\beta\}$ & 0 & 60, 63\\
58 & $\{\eta,\delta,\xi\}$ & 1 & 60\\
59 & $\{\zeta\}$ & 0 & 61, 63\\
60 & $\{\eta,\delta,\beta\}$ & 1 & 64\\
61 & $\{\zeta,\delta\}$ & 1 & 62, 64\\
62 & $\{\rho\}$ & 1 & 65\\
63 & $\{s\}$ & 0 & 64\\
64 & $\{\delta,s\}$ & 1 & 65\\
65 & $\{\chi\}$ & 1 & 66\\
66 & $\{\tau\}$ & 1 & ---\\
\end{longtable}
}

\end{document}